\documentclass[12pt,oneside,reqno]{amsart}
\usepackage{graphicx}

\graphicspath{ {figs/} }

\usepackage{pict2e}
\usepackage{amssymb}
\usepackage{amsthm}                
\usepackage[margin=0.9in]{geometry}  
\usepackage{enumitem}
\usepackage{epstopdf}
\usepackage{amsmath}
\usepackage{subfigure}
\usepackage{latexsym}
\usepackage{amsmath}
\usepackage{amsfonts}
\usepackage{amssymb}
\usepackage{tikz}
\usepackage{mathdots}
\usepackage{comment}
\usepackage{float}
\usepackage[mathscr]{euscript}
\usepackage{epsfig}
\usepackage{caption}
\usepackage{enumitem}

\usetikzlibrary{arrows.meta}
\usetikzlibrary{knots}
\usetikzlibrary{hobby}
\usetikzlibrary{arrows,decorations.markings}

\usepackage[bookmarks=true,bookmarksnumbered=true, breaklinks=true,pdfstartview=FitH,hyperfigures=false,plainpages=false, naturalnames=true,colorlinks=true,pagebackref=true,pdfpagelabels]{hyperref}

\usepackage[utf8]{inputenc}
\usepackage{longtable}
\usepackage[lite]{amsrefs}

\renewcommand{\PrintDOI}[1]{\href{http://dx.doi.org/\detokenize{#1}}{doi: \detokenize{#1}}%
	\IfEmptyBibField{pages}{, (to appear in print)}{}}

\theoremstyle{definition}
\newtheorem{theorem}{Theorem}[section]
\newtheorem{lemma}[theorem]{Lemma}
\newtheorem{corollary}[theorem]{Corollary}
\newtheorem{proposition}[theorem]{Proposition}

\theoremstyle{definition}
\newtheorem{definition}[theorem]{Definition}
\newtheorem{example}[theorem]{Example}
\theoremstyle{remark}
\newtheorem{remark}[theorem]{Remark}
\numberwithin{equation}{section}

\numberwithin{equation}{section}

\title{Singquandle Shadows and their Invariants}
\title{Singquandle Shadows and Singular knot Invariants}

\author{Jose Ceniceros}
\address{Hamilton College, Clinton, NY }
\email{jcenicer@hamilton.edu}

\author{Indu R. Churchill}
\address{State University of New York at Oswego, Oswego, NY }
\email{indurasika.churchill@oswego.edu}

\author{Mohamed Elhamdadi}
\thanks{M.E. was partially supported by Simons Foundation collaboration grant 712462.}
\address{University of South Florida, Tampa, FL }
\email{emohamed@math.usf.edu}

\date{}
\dedicatory{}

\begin{document}
\maketitle

\begin{abstract}
We introduce shadow structures for singular knot theory.  Precisely, we define \emph{two} invariants of singular knots and links.  First, we introduce a notion of action of a singquandle on a set to define a shadow counting invariant of singular links which generalize the classical shadow colorings of knots by quandles. We then define a shadow polynomial invariant for shadow structures. Lastly, we enhance the shadow counting invariant by combining both the shadow counting invariant and the shadow polynomial invariant.  Explicit examples of computations are given.   
\end{abstract}

\tableofcontents

\section{Introduction}
Quandles are non-associative algebraic structures that are modeled on the three Reidemester moves in classical knot theory.  Thus, they are appropriate algebraic structures for constructing invariants of knots and links in $3$-space and knotted surfaces in $4$-space.  Quandles were introduced independently by Joyce \cite{Joyce} and Matveev \cite{Matveev} in the 1980s.  Quandles have been investigated in many areas of mathematics such as quasigroups and Moufang loops \cite{E}, the Yang-Baxter equation \cite{CES,CN}, representation theory \cite{EM}, and ring theory \cite{EFT}.  For more information on quandles, the reader is advised to consult the book \cite{EN}.  Knot theory has been extended in several directions, for example, singular knot theory.  In \cite{BL}, connections between Jones type invariants defined in \cite{Jones} and Vassiliev invariants of singular knots defined in \cite{V} were established.  In \cite{JL}, a variation of the Hecke algebra was used to construct a Jones-type invariant for singular knots.  Combinatorial singular knot theory has the so-called generalized Reidemeister moves \cite{JL}.  These generalized moves were used in \cite{CEHN} to extend the concept of quandle to an algebraic structure called \emph{singquandle} to provide invariants for singular knots.  In \cite{BEHY}, generating sets of the generalized Reidemeister moves for oriented singular links were introduced and used to distinguish singular knots and links. In \cite{CCEH}, the quandle cocycle invariant \cite{CJKLS} was extended to oriented singular knots and used to construct a state sum invariant for singular links. Furthermore in \cite{CCE}, the quandle polynomial invariant was extended to the case of singquandles. A singular link invariant  was constructed from the singquandle polynomial and it was shown to generalize the singquandle counting invariant in \cite{CCE}. 

The article is organized as follows.  In Section~\ref{review}, the basics of quandle theory are reviewed, and some examples are given.  Section~\ref{OSC} provides a review of the basic constructions of oriented singquandles as well as the singquandle counting invariant. In Section ~\ref{singpoly}, we discuss the singquandle polynomial and the subsingquandle polynomial, which we introduced in a previous paper \cite{CCE} and used to define a singular link invariant. Section~\ref{SS} defines \emph{singquandle shadows} which is used to generalize the \emph{shadow colorings} of knot diagrams by quandles previously defined in \cites{CN}. Furthermore, the \emph{shadow singquandle polynomial} and the \emph{singquandle shadow polynomial invariant} for a singular link $L$ is defined. In Section~\ref{enhanc}, the shadow counting invariant is used to define an enhanced shadow link invariant by combining the shadow singquandle counting invariant and the shadow singquandle polynomial. Lastly, Section~\ref{examples} examines the strength of the singquandle shadow polynomial invariant. Two examples are provided to illustrate that the singquandle shadow polynomial is sensitive to differences in singular links not detected by the singquandle coloring invariant and singquandle polynomial invariant.

\section{Basics of Quandles}\label{review}

In this paper we will consider only finite quandles and singquandles. 
We'll review the basics of quandles; more details on the topic can be found in  \cites{EN, Joyce, Matveev}.
\begin{definition}\label{quandle}
A set $X$ with binary operation $\ast$ is called a \emph{quandle} if the following three identities are satisfied.
\begin{eqnarray*}
& &\mbox{\rm (i) \ }   \mbox{\rm  For all $x \in X$,
$x* x =x$.} \label{axiom1} \\
& & \mbox{\rm (ii) \ }\mbox{\rm For all $y,z \in X$, there is a unique $x \in X$ such that 
$ x*y=z$.} \label{axiom2} \\
& &\mbox{\rm (iii) \ }  
\mbox{\rm For all $x,y,z \in X$, we have
$ (x*y)*z=(x*z)*(y*z). $} \label{axiom3} 
\end{eqnarray*}
\end{definition}
From Axiom (ii) of Definition~\ref{quandle} we can write the element $x$ as $z \bar{*} y = x$. Notice that this operation $\bar{*}$ defines a quandle structure on $X$. The axioms of a quandle correspond respectively to the three Reidemeister moves of types I, II and III (see \cite{EN} for examples).  In fact, one of the motivations of defining quandles came from knot diagrammatic.

 A {\it quandle homomorphism} between two quandles $(X,*)$ and $(Y,\triangleright)$ is a map $f: X \rightarrow Y$ such that $f(x *y)=f(x) \triangleright f(y) $, where
 $*$ and $\triangleright$ 
 denote respectively the quandle operations of $X$ and $Y$. Furthermore, if $f$ is a bijection, then it is called a 
{\it quandle isomorphism} between $X$ and $Y$. 



\noindent Some typical examples of quandles:
\begin{itemize}
\item
Any non-empty set $X$ with the operation $x*y=x,$ for all $x,y \in X,$ is
a quandle called a  {\it trivial} quandle.
\item
Any group $X=G$ with conjugation $x*y=y^{-1} xy$ is a quandle.

\item
Let $G$ be an abelian group.
For elements  
$x,y \in G$, 
define
$x*y \equiv 2y-x$.
Then $\ast$ defines a quandle
structure on $G$ called \emph{Takasaki} quandle.  In case $G=\mathbb{Z}_n$ (integers mod $n$) the quandle is called {\it dihedral quandle}.
This quandle can be identified with  the
set of reflections of a regular $n$-gon
  with conjugation
as the quandle operation.
\item
Any $\Lambda = (\mathbb{Z }[T^{\pm 1}])$-module $M$
is a quandle with
$x*y=Tx+(1-T)y$, $x,y \in M$, called an {\it  Alexander  quandle}.

\end{itemize}

\vspace{0.5cm}

\section{Oriented Singquandles and the Counting Invariant}\label{OSC}

We will provide a basic overview of an oriented singquandle as well as the singquandle counting invariant. For a detailed construction of oriented singquandle and the singquandle counting invariant, see \cites{BEHY, CEHN}. We will be adopting 
the following conventions at classical and singular crossings.

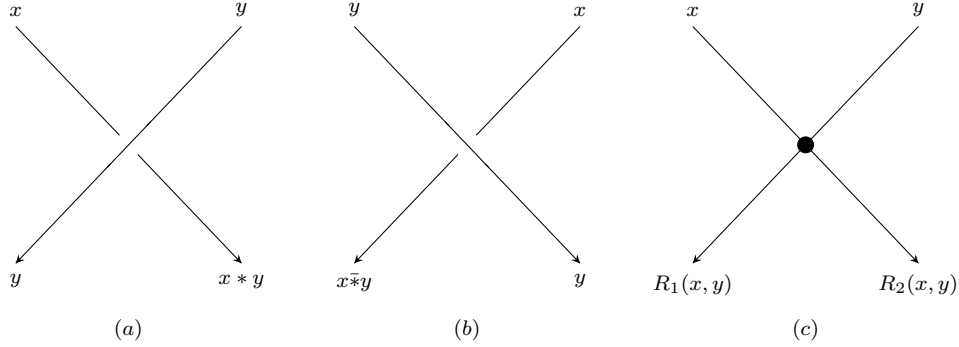
\begin{figure}[h]
\begin{tikzpicture}[use Hobby shortcut, scale=1.5, add arrow/.style={postaction={decorate}, decoration={
  markings,
  mark=at position 1 with {\arrow[scale=1,>=stealth]{>}}}}]
\begin{knot}[
consider self intersections=true,
clip width=5,
 ignore endpoint intersections=false,
]
\strand (1,1) ..(-1,-1.1)[add arrow]; 
\strand (-1,1) ..(1,-1.1)[add arrow];
    
\strand (2,1) ..(4,-1.1)[add arrow];
\strand (4,1) ..(2,-1.1)[add arrow];  
    
\draw (5,1) ..(7,-1.1)[add arrow];
\draw (7,1) ..(5,-1.1)[add arrow];  
\end{knot}
\node[circle,draw=black, fill=black, inner sep=0pt,minimum size=6pt] (a) at (6,-.05) {};
\node[above] at (-1,1) {\tiny $x$};
\node[above] at (1,1) {\tiny $y$};
\node[below] at (1,-1.1) {\tiny $x*y$};
\node[below] at (-1,-1.1) {\tiny $y$};
\node[above] at (2,1) {\tiny $y$};
\node[above] at (4,1) {\tiny $x$};
\node[below] at (2,-1.1) {\tiny $x\bar{*}y$};
\node[below] at (4,-1.1) {\tiny $y$};
\node[above] at (7,1) {\tiny $y$};
\node[above] at (5,1) {\tiny $x$};
\node[below] at (5,-1.1) {\tiny $R_1(x,y)$};
\node[below] at (7,-1.1) {\tiny $R_2(x,y)$};
\node[below] at (0,-1.5) {\tiny $(a)$};
\node[below] at (3,-1.5) {\tiny $(b)$};
\node[below] at (6,-1.5) {\tiny $(c)$};
\end{tikzpicture}
\caption{(a) Coloring of arcs at a positive crossing, (b) coloring of arcs at a negative crossing, (c) colorings of semi-arcs at a singular crossing.}
\label{crossings}
\end{figure}

Generating sets of oriented singular Reidemeister moves were studied and were used to define oriented singquandles in \cite{BEHY}. We will follow the naming convention for oriented singular Reidemeister moves used in \cite{BEHY}. Note that if we let $(S, *)$ be a quandle, we only need to consider the colorings from singular Reidemeister moves in Figures~\ref{The generalized Reidemeister move RIV}, \ref{The generalized Reidemeister move RIVb}, and \ref{The generalized Reidemeister move RV}.
\begin{figure}[h] 
\begin{tikzpicture}[use Hobby shortcut,, scale=1, add arrow/.style={postaction={decorate}, decoration={
  markings,
  mark=at position 1 with {\arrow[scale=1,>=stealth]{>}}}}]
\begin{knot}[
consider self intersections=true,
clip width=4,
 ignore endpoint intersections=false,
 flip crossing/.list={2,3}
]
 	\strand (-1,2) .. (2,-2)[add arrow];
	\strand (2,2) ..  (-1,-2)[add arrow];
	\strand  (.7,-2.4)..(0,-1.5).. (-1,0)..(0,1.5)..(.7,2.4)[add arrow];
\end{knot}

	\node[left] at (-1,2) {\tiny $a$};
	\node[right] at (0,.8){\tiny$a\bar{*}b$};
    \node[left] at (-1,-2) {\tiny $R_1(a\bar{*}b,c)*b$};
    
    \node[right] at (2,2) {\tiny $c$};
    \node[left] at (.5,-2) {\tiny$b$};
    \node[right] at (2,-2) {\tiny $R_2(a \bar{*}b,c)$};
    \node[circle,draw=black, fill=black, inner sep=0pt,minimum size=6pt] (a) at (.5,0) {};
    \draw [very thick, <->] (3,0) -- (4,0);

\begin{knot}[
consider self intersections=true,
clip width=4,
 ignore endpoint intersections=false,
flip crossing/.list={2,3}
]
	\strand (8,2) .. (5,-2)[add arrow];   
    \strand (5,2).. (8,-2)[add arrow];
    \strand (6.3,-2.4).. (7,-1.5).. (8,0).. (7,1.5)..(6.3,2.4)[add arrow];
  \end{knot}  
    \node[left] at (5,2) {\tiny $a$};
    \node[left] at (8,2) {\tiny $c$};
    \node[left] at (6.3,-2) {\tiny $b$};

    \node[left] at (7,.8) {\tiny$c*b$};
    \node[left] at (5.5,-1.4) {\tiny $R_1(a,c*b)$};
    \node[right] at (8,-2) {\tiny $R_2(a,c*b)\bar{*}b$};
    \node[circle,draw=black, fill=black, inner sep=0pt,minimum size=6pt] (a) at (6.5,0) {};
\end{tikzpicture}
\vspace{.2in}
		\caption{The Reidemeister move $\Omega 4a$ with colors.}
		\label{The generalized Reidemeister move RIV}
\end{figure}
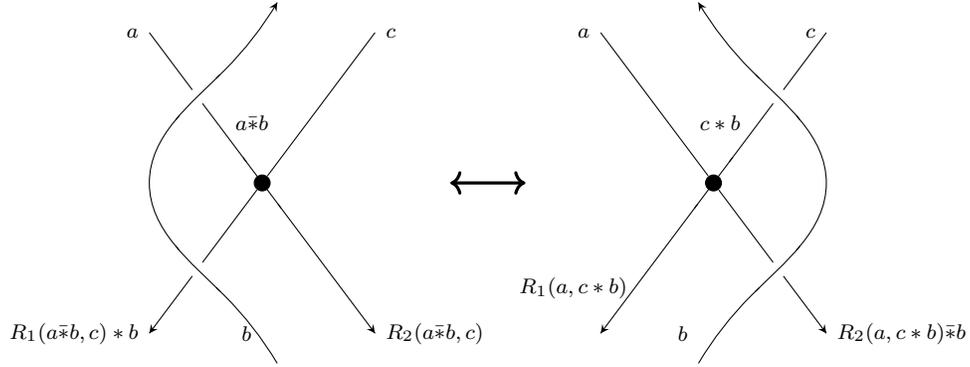

\begin{figure}[h]
\begin{tikzpicture}[use Hobby shortcut,, scale=1, add arrow/.style={postaction={decorate}, decoration={
  markings,
  mark=at position 1 with {\arrow[scale=1,>=stealth]{>}}}}]
\begin{knot}[
consider self intersections=true,
clip width=4,
 ignore endpoint intersections=false,
]
 	\strand (-1,2) .. (2,-2)[add arrow];
	\strand (2,2) ..  (-1,-2)[add arrow];
	\strand  (.7,-2.4)..(0,-1.5).. (-1,0)..(0,1.5)..(.7,2.4)[add arrow];
\end{knot}

	\node[left] at (-1,2) {\tiny $a$};
	\node[left] at (-1,0){\tiny$b \bar{*} R_1(a,c)$};
    \node[left] at (-1,-2) {\tiny $R_1(a,c)$};
    \node[above] at (.7,2.4) {\tiny$(b \bar{*} R_1(a,c))*a$};
    \node[right] at (2,2) {\tiny $c$};
    \node[left] at (.5,-2) {\tiny$b$};
    \node[right] at (2,-2) {\tiny $R_2(a,c)$};
    \node[circle,draw=black, fill=black, inner sep=0pt,minimum size=6pt] (a) at (.5,0) {};
    \draw [very thick, <->] (3,0) -- (4,0);

\begin{knot}[
consider self intersections=true,
clip width=4,
 ignore endpoint intersections=false,
]
	\strand (8,2) .. (5,-2)[add arrow];   
    \strand (5,2).. (8,-2)[add arrow];
    \strand (6.3,-2.4).. (7,-1.5).. (8,0).. (7,1.5)..(6.3,2.4)[add arrow];
  \end{knot}

    \node[left] at (5,2) {\tiny $a$};
    \node[left] at (8,2) {\tiny $c$};
    \node[left] at (6.3,-2) {\tiny $b$};

    \node[right] at (8,0) {\tiny$b*R_2(a,c)$};
    \node[left] at (5.5,-1.4) {\tiny $R_1(a,c)$};
    \node[right] at (8,-2) {\tiny $R_2(a,c)$};
    \node[above] at (6.3,2.4) {\tiny $(b*R_2(a,c))\bar{*}c$};
    \node[circle,draw=black, fill=black, inner sep=0pt,minimum size=6pt] (a) at (6.5,0) {};
\end{tikzpicture}
\vspace{.2in}
		\caption{The Reidemeister move $\Omega 4e$ with colors.}
		\label{The generalized Reidemeister move RIVb}
\end{figure}
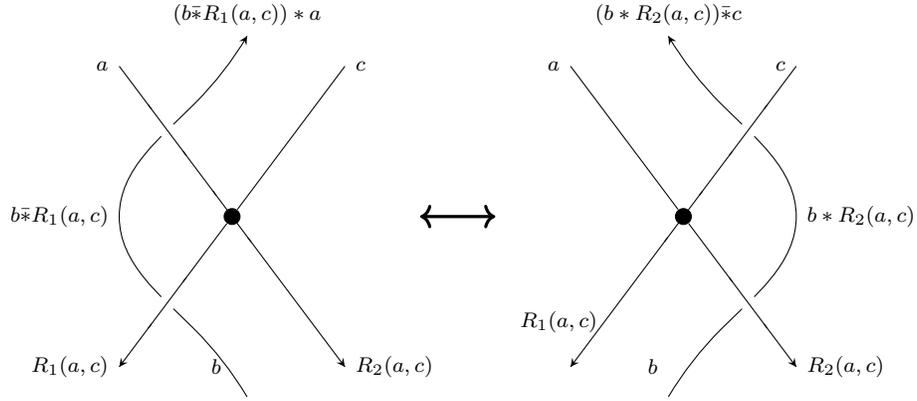

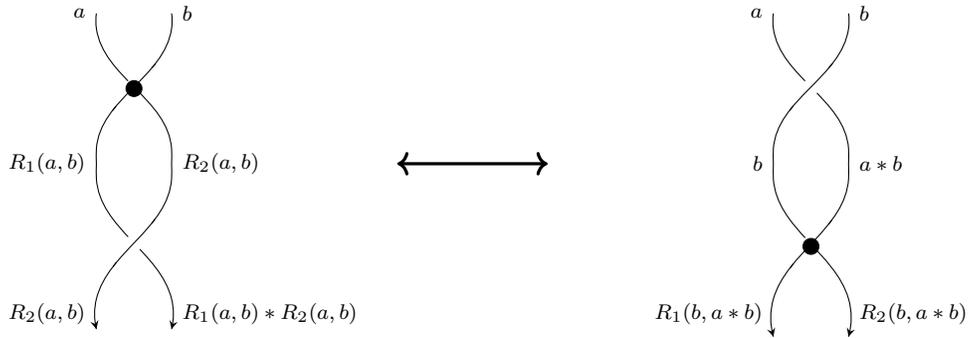
\begin{figure}[h]
\begin{tikzpicture}[use Hobby shortcut,, scale=1, add arrow/.style={postaction={decorate}, decoration={
  markings,
  mark=at position 1 with {\arrow[scale=1,>=stealth]{>}}}}]
\begin{knot}[
consider self intersections=true,
clip width=4,
 ignore endpoint intersections=false,
flip crossing/.list={2}
]
	\strand (-6,2) .. (-6,1.8).. (-7,.2) .. (-7,0)..(-7,-.2) .. (-6,-1.8).. (-6,-2.2)[add arrow];
		
 	\strand (-7,2).. (-7,1.8).. (-6,.2).. (-6,0)..(-6,-.2).. (-7,-1.8)..(-7,-2.2)[add arrow] ;
\end{knot}
	\node[left] at (-7,2) {\tiny $a$}; 
    \node[left] at (-7,0) {\tiny $R_1(a,b)$};
    \node[left] at (-7,-2) {\tiny $R_2(a,b)$};
	\node[right] at (-6,2) {\tiny $b$};
	\node[right] at (-6,0) {\tiny $R_2(a,b)$};
	\node[right] at (-6,-2) {\tiny $R_1(a,b)*R_2(a,b)$};
\node[circle,draw=black, fill=black, inner sep=0pt,minimum size=6pt] (a) at (-6.5,1) {};
	
    \draw [very thick, <->] (-3,0) -- (-1,0);

\begin{knot}[
consider self intersections=true,
clip width=4,
 ignore endpoint intersections=false,
flip crossing/.list={1}
]
	\strand (2,2) .. (2,1.8).. (3,.2) .. (3,0)..(3,-.2) .. (2,-1.8).. (2,-2.3)[add arrow];
	\strand (3,2).. (3,1.8).. (2,.2).. (2,0)..(2,-.2).. (3,-1.8)..(3,-2.3)[add arrow];
	\end{knot}
	
	\node[left] at (2,2) {\tiny $a$}; 
	\node[left] at (2,0) {\tiny $b$};
	\node[left] at (2,-2) {\tiny $R_1(b,a*b)$};
	\node[right] at (3,2) {\tiny $b$};
	\node[right] at (3,0) {\tiny $a * b$};
	\node[right] at (3,-2) {\tiny $R_2(b,a*b)$};
	\node[circle,draw=black, fill=black, inner sep=0pt,minimum size=6pt] (a) at (2.5,-1.1) {};
\end{tikzpicture}
\vspace{.2in}
		\caption{The Reidemeister move $\Omega 5a$ with colors.}
		\label{The generalized Reidemeister move RV}
\end{figure}

\begin{definition}\label{oriented SingQdle}
	Let $(X, *)$ be a quandle.  Let $R_1$ and $R_2$ be two maps from $X \times X$ to $X$.  The triple $(X, *, R_1, R_2)$ is called an {\it oriented singquandle} if the following axioms are satisfied for all $a,b,c \in X$:
	\begin{eqnarray}
		R_1(a\bar{*}b,c)*b&=&R_1(a,c*b)  \label{eq1}\\
		R_2(a\bar{*}b, c) & =&  R_2(a,c*b)\bar{*}b \label{eq2}\\
	      (b\bar{*}R_1(a,c))*a   &=& (b*R_2(a,c))\bar{*}c  \label{eq3}\\
R_2(a,b)&=&R_1(b,a*b)   \label{eq4}\\
R_1(a,b)*R_2(a,b)&=&R_2(b,a*b).   \label{eq5}	
\end{eqnarray}	
\end{definition}
\begin{remark}
We will only consider oriented singquandles in this paper. Therefore, we will simply refer to oriented singquandle as singquandles. For a description of unoriented singquandles, see \cite{CEHN}.
\end{remark}
In \cite{CCEH}, the following family of singquandles was introduced. 
\begin{example}\label{singfam}
 Let $n$ be a positive integer, let $a$ be an invertible element in $\mathbb{Z}_n$ and let $b,c \in \mathbb{Z}_n$, then the binary operations $x*y = ax+(1-a)y$, $ R_1(x,y) = bx + cy$ and $R_2(x,y)= acx + [b+ c(1 - a)]y $ make the quadruple $(\mathbb{Z}_n,*, R_1,R_2)$ into an oriented singquandle.
\end{example}
The singquandles used in this paper are obtained from Example~\ref{singfam} with specific values of $a,b,$ and $c$.
\begin{definition}
Let $(X, *, R_1, R_2)$ be a singquandle.  A subset $M \subset X$ is called a subsingquandle if $(M, *, R_1, R_2)$ is itself a singquandle.  In particular, $M$ is closed under the operations $*, R_1$ and $R_2$.  
\end{definition}
\noindent We can define the notion of a homomorphism and isomorphism of oriented singquandles.

\begin{definition}\label{singHom}
A map $f: X \rightarrow Y$ is called a \emph{homomorphism} of oriented singquandles $(X, *, R_1, R_2)$ and $(Y, \triangleright, R'_1, R'_2)$ if the following conditions are satisfied for all $x,y \in X$
\begin{eqnarray}
f(x*y)&=&f(x) \triangleright f(y)\label{3.6}\\
f(R_1(x,y))&=&R'_1(f(x),f(y))\label{3.7}\\
f(R_2(x,y))&=&R'_2(f(x),f(y)).\label{3.8}
\end{eqnarray}
An oriented singquandle \emph{isomorphism} is a bijective oriented singquandle homomorphism. We say two oriented singquandles are \emph{isomorphic} if there exists an oriented singquandle isomorphism between them.
\end{definition}

The authors of this paper introduced the idea of a fundamental singquandle associated to a singular link $L$, denoted by $\mathcal{SQ}(L)$, in \cite{CCE}. Therefore, for any oriented singular link $L$ and an oriented singquandle $(S,*,R_1', R_2')$, the set of singquandle homomorphism from $(\mathcal{SQ}(L),\triangleright ,R_1,R_2)$ to $(S, *, R_1', R_2')$ is defined by:
\begin{equation*}
\begin{split}
\textup{Hom}(\mathcal{SQ}(L),S) = \lbrace f \, : \,  &\mathcal{SQ}(L) \rightarrow S \, \vert \, f( x\triangleright y) = f(x) * f(y),\\ &f(R_1(x,y))= R_1'(f(x),f(y)),f(R_2(x,y))= R_2'(f(x),f(y)) \rbrace.
\end{split}
\end{equation*}
The set defined above was shown to be an invariant of oriented singular links in \cite{BEHY}. Furthermore, this set can be used to define computable invariants of oriented singular links. For example, by taking the cardinality of $\textup{Hom}(\mathcal{SQ}(L),S)$, we obtain the following invariant of oriented singular links.

\begin{definition}
Let $L$ be an oriented singular link and $(S,*,R_1,R_2)$ be an oriented singquandles. The \emph{singquandle counting invariant} is 
\[ \#\textup{Col}_S (L)=\vert \textup{Hom}(\mathcal{SQ}(L),S) \vert.  \]
\end{definition}

\begin{remark}\label{imagesubsing}
We also note that the image, $\textup{Im}(f)$, for each $f \in\textup{Hom}(\mathcal{SQ}(L),S)$ is a subsingquandle of $S$ as shown in \cite{CCE}.
\end{remark}

\section{Review of the Singquandle Polynomial}\label{singpoly}
The quandle polynomial was introduced in \cite{N} and generalized in \cite{N2}.  In \cite{CCE}, the authors of this paper defined the singquandle polynomial, the subsingquandle polynomial and a polynomial invariant of singular links. In this section, we will give an overview of the construction of the singquandle polynomial, the subsingquandle polynomial, and the polynomial invariant of singular links. We will follow the construction and notation introduced in \cite{CCE}.

\begin{definition}
Let $(X,*,R_1,R_2)$ be a finite singquandle. For every $x \in X$, define
\[ C^1(x) = \lbrace y \in X \, \vert \, y * x = y \rbrace \quad \text{and} \quad R^1(x) = \lbrace y \in X \, \vert \, x * y = x \rbrace, \]
\[ C^2(x) = \lbrace y \in X \, \vert \, R_1(y , x) = y \rbrace \quad \text{and} \quad R^2(x) = \lbrace y \in X \, \vert \, R_1(x , y) = x \rbrace, \]
\[ C^3(x) = \lbrace y \in X \, \vert \, R_2(y , x) = y \rbrace \quad \text{and} \quad R^3(x) = \lbrace y \in X \, \vert \, R_2(x , y) = x \rbrace. \]\\
Let $c^i(x) = \vert C^i(x)\vert$ and $r^i(x) = \vert R^i(x)\vert$ for $i=1,2,3$. Then the \emph{singquandle polynomial of X} is 

\[ sqp(X) = \sum_{x\in X} s_1^{r^1(x)}t_1^{c^1(x)}s_2^{r^2(x)}t_2^{c^2(x)}s_3^{r^3(x)}t_3^{c^3(x)}.  \]
\end{definition}
We note that the value $r^i(x)$ is the number of elements in $X$ that act trivially on $x$, while $c^i(x)$ is the number of elements of $X$ on which $x$ acts trivially via $*, R_1$ and $R_2.$ Furthermore, if $Y \subset X$ is a subsingquandle we can define the following singquandle polynomial for $Y$ as a subsignquandle of $X$

\begin{definition}
Let $(X,*, R_1,R_2)$ be a finite singquandle and $S \subset  X$ a subsingquandle. Then the \emph{subsingquandle polynomial} is 
\[ Ssqp(S \subset X ) = \sum_{x \in S} s_1^{r^1(x)}t_1^{c^1(x)}s_2^{r^2(x)}t_2^{c^2(x)}s_3^{r^3(x)}t_3^{c^3(x)}. \]
\end{definition} 
The subsingquandle polynomials can be thought of as the contributions to the singquandle polynomial coming from the subsingquandles we are considering. Using the subsingquandles polynomial we can now define the following polynomial invariant of singular links.

\begin{definition}
Let $L$ be a singular link, $(X,*,R_1,R_2)$ a finite singquandle. Then the multiset
\[ \Phi_{Ssqp}(L,X) = \lbrace Ssqp(Im(f) \subset X) \, \vert \, f \in \text{Hom}(\mathcal{SQ}(L), X\rbrace \]
is the \emph{subsingquandle polynomial invariant of $L$} with respect to $X$. We can also represent this invariant in the following polynomial-style form by converting the multiset elements to exponents of a formal variable $u$ and converting their multiplicities to coefficients:
\[ \phi_{Ssqp}(L,X) = \sum_{f \in \textup{Hom}(\mathcal{SQ}(L),X)} u^{Ssqp(Im(f) \subset X)}.\]
\end{definition}

\begin{example}
Consider the \emph{2-bouquet graphs of type $L$} listed as $1^l_1$ in \cite{Oyamaguchi}. Let $(S,*,R_1,R_2)$ be the singquandle with $S=\mathbb{Z}_4$ and operations $x*y = 3x-2y = x\bar{*}y$, $R_1(x,y)=2 x + 3 y$, and $R_2(x,y)=x$. 
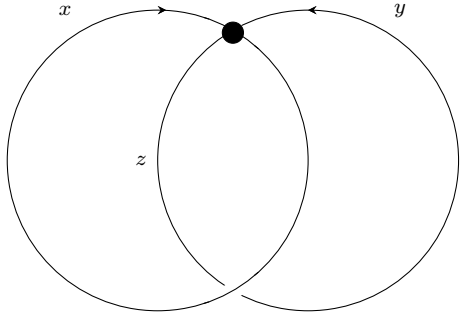
\begin{figure}[h]
\begin{tikzpicture}[use Hobby shortcut]
\begin{knot}[
 clip width=4
]

\strand[decoration={markings,mark=at position .25 with
    {\arrow[scale=1,>=stealth]{<}}},postaction={decorate}](-1,0) circle[radius=2cm];
\strand[decoration={markings,mark=at position .25 with
    {\arrow[scale=1,>=stealth]{>}}},postaction={decorate}] (1,0) circle[radius=2cm];

\end{knot}

\node[circle,draw=black, fill=black, inner sep=0pt,minimum size=8pt] (a) at (0,1.7) {};
\node[left] at (-1,0) {\tiny $z$};
\node[left] at (-2,2) {\tiny $x$};
\node[right] at (2,2) {\tiny $y$};
\end{tikzpicture}
\vspace{.2in}
		\caption{Diagram of $1_1^l$.}
		\label{11l}
\end{figure}
We can identify each coloring of $1_1^l$ by $S$ with the triple $(f(x),f(y),f(z))$. Using the fact that $z = R_1(x,y), x = R_2(x,y)$ and $z * x =y,$  by a straightforward computation we  obtain the following colorings: 
\begin{equation*}
\begin{split}
\textup{Hom}(\mathcal{SQ}(1_1^l),S)=\{(1, 1, 1), (1, 2, 0), (1, 3, 3), (1, 0, 2), (2, 1, 3), (2, 2, 2), (2, 3, 1), (2, 0, 0),\\ (3, 1, 1), (3, 2, 0), (3, 3, 3), (3, 0, 2), (0, 1, 3), (0, 2, 2), (0, 3, 1), (0, 0, 0)\}.
\end{split}
\end{equation*}

Therefore, $\#\textup{Col}_S (1_1^l)=16$. We compute $r^i$ and $c^i$ for $i=1,2,3$:

\[
\begin{tabular}{ r|  c  c  }
$x$ & $r^1(x)$ & $c^1(x)$\\
  \hline			
1&  2 & 2 \\
2&  2 & 2  \\
3&  2 & 2 \\
0&  2 & 2  \\
\end{tabular}
\qquad
\begin{tabular}{ r|  c  c  }
$x$ & $r^2(x)$ & $c^2(x)$\\
  \hline			
1&  1 & 1 \\
2&  1 & 1  \\
3&  1 & 1 \\
0&  1 & 1  \\
\end{tabular}
\qquad
\begin{tabular}{ r|  c  c  }
$x$ & $r^3(x)$ & $c^3(x)$\\
  \hline			
1&  4& 4\\
2&  4 & 4  \\ 
3&  4 & 4 \\
0&  4 & 4  \\
\end{tabular}.
\]
In order to compute the subsingquandle polynomial invariant of $1_1^l$ we consider the corresponding subsingquandle for each coloring. Therefore, we obtain  $$\phi_{Ssqp}(1_1^l,S) = 4 u^{s_1^2 t_1^2 s_2  t_2 s_3^4 t_3^4}+4 u^{2 s_1^2 t_1^2 s_2 t_2 s_3^4 t_3^4}+8
   u^{4 s_1^2 t_1^2 s_2 t_2 s_3^4 t_3^4}.$$
\end{example}
\section{Singquandle Shadows}\label{SS}

In this section, we define \emph{singquandle shadows} which can be used to generalize the \emph{shadow colorings} of knot diagrams by quandles previously defined in \cites{CN}.

\begin{definition}
Let $(S,*,R_1,R_2)$ be a singquandle. An \emph{S-set} is a set  $X$ and a map $\cdot : X \times S \rightarrow X$ satisfying the following conditions: 
\begin{enumerate}[label=(\roman*)]
    \item For all $s \in S$, $\cdot s : X \rightarrow X$ mapping $x$ to $x \cdot s$ is a bijection.
    \item For all $s_1, s_2 \in S$ and $x \in X$,
 \begin{eqnarray}
(x \cdot s_1)\cdot s_2 &=& (x \cdot s_2)\cdot ( s_1 * s_2)\\
(x \cdot s_1) \cdot s_2 &=& (x \cdot R_1(s_1,s_2)) \cdot R_2(s_1,s_2).
\end{eqnarray}
\end{enumerate}

\end{definition}
The meaning of these two equations will become clear from Figure~\ref{shadowX}.
\begin{definition}
A \emph{singquandle shadow} or $S$-shadow is the pair of an oriented sinquandle $(S,*,R_1,R_2)$ and a $S$-set $(X,\cdot)$, denoted by $(S,X,*,R_1,R_2,\cdot)$ or simply by $(S,X)$. Let $S^{\prime}$ be a subsingquandle of $S$. A subset $Y$ of $X$ closed under the action of $S^{\prime}$ is an \emph{subshadow} of $(S,X)$, which we will denote by $(S^\prime, Y) \subset (S,X)$.
\end{definition}

The following definition will allow us to present a shadow operation in an alternate form that will be useful in later sections.

\begin{definition}\label{shadowmatrix}
When $(X,S)$ is a singquandle shadow with $X$ and $S$ finite, the \emph{shadow matrix} of the singquandle shadow $(X = \lbrace x_1, \dots, x_m \rbrace$, $S = \lbrace s_1,\dots, s_n \rbrace)$ is the $m \times n$ matrix whose $(i,j)$ entry is $k$ where $x_k = x_i \cdot s_j$.
\end{definition}
 
Let $(S,*,R_1,R_2)$ and $(S',\triangleright, R'_1,R'_2)$ be singquandles. Furthermore, let $(X, \cdot)$ be an $S$-set and $(X', \bullet)$ be an $S'$-set.  A map $f: S \rightarrow S'$ makes $X'$ inherit a natural action of $S$ via the map $f$ by $x' \cdot s :=x' \bullet f(s)$.

\begin{definition}
A \emph{homomorphism of sinquandle shadows} between $(S,X,*,R_1,R_2,\cdot)$ and \\ $(S',Y,\triangleright,R'_1,R'_2,\bullet)$ is a pair of maps $\phi:(X, \cdot) \rightarrow (Y,\bullet)$ and $f:(S,*,R_1,R_2) \rightarrow (S',\triangleright,R'_1,R'_2)$, such that $f$ is a singquandle homomorphism, that is the identities (\ref{3.6}), (\ref{3.7}) and (\ref{3.8}) are satisfied and for all $x \in X$ and $s \in S$, we have
\begin{eqnarray}\label{action}
\phi(x \cdot s)=\phi(x) \bullet f(s).
\end{eqnarray}
Furthermore, if $\phi$ and $f$ are bijections then we have a \emph{singquandle shadow isomorphism}.
\end{definition}

From this definition it is straightforward to obtain the following lemma.

\begin{lemma}
$(Im(f),Im(\phi),\triangleright,R'_1,R'_2,\bullet)$ is a \emph{subshadow} of $(S',Y,\triangleright,R'_1,R'_2,\bullet)$.
\end{lemma}

\begin{example}\label{shadow4elem}
Let $(S,*,R_1,R_2)$ be an oriented singquandle with $S=\mathbb{Z}_8=\{ 1,2,3,4,5,6,7,0 \}$, $x*y = 5x-4y = x\bar{*}y$, $R_1(x,y) =3x+4y $, and $R_2(x,y)= 4x+3y$. Then the four element set $X= \mathbb{Z}_4=\{ 1,2,3,0 \}$ with map $\cdot s : \mathbb{Z}_4 \rightarrow \mathbb{Z}_4$ for each $s \in S$ defined by $x\cdot s = x + 2 s + s^2$ is a singquandle shadow. Note that $\cdot$ has the following operation table
\[ \begin{array}{r|cccccccc}
\cdot & 1 & 2 & 3 & 4 & 5 & 6 & 7 & 0\\
\hline
 1& 0 & 1 & 0 & 1 & 0 & 1 & 0 & 1 \\
 2& 1 & 2 & 1 & 2 & 1 & 2 & 1 & 2 \\
 3& 2 & 3 & 2 & 3 & 2 & 3 & 2 & 3 \\
 0& 3 & 0 & 3 & 0 & 3 & 0 & 3 & 0 \\
\end{array}. \]
Furthermore, by Definition~\ref{shadowmatrix} the shadow operation $\cdot$ can be presented by the shadow matrix,
\[
\left[
\begin{array}{cccccccc}
 0 & 1 & 0 & 1 & 0 & 1 & 0 & 1 \\
 1 & 2 & 1 & 2 & 1 & 2 & 1 & 2 \\
 2 & 3 & 2 & 3 & 2 & 3 & 2 & 3 \\
 3 & 0 & 3 & 0 & 3 & 0 & 3 & 0 \\
\end{array}
\right].
\]
\end{example}

\begin{example}\label{shadowsing}
Let $(S,*,R_1,R_2)$ be any oriented singquandle and let $X = S$. Then $X$ is a singquandle shadow under the shadow operation $x \cdot s = x * s$ for all $x,y \in S$, since we have 
\begin{eqnarray}
(x \cdot s_1) \cdot s_2 &=& (x * s_1)*s_2 \nonumber\\
&=& (x * s_2) * (s_1 * s_2) \nonumber\\
&=& (x\cdot s_2) \cdot (s_1 \cdot s_2), \label{Eq1}
\end{eqnarray}
and
\begin{eqnarray}
(x \cdot s_1) \cdot s_2 &=& (x * s_1)*s_2 \nonumber\\
&=& (x * R_1 (s_1,s_2)) * R_2(s_1,s_2) \nonumber\\
&=& (x \cdot R_1 (s_1,s_2)) \cdot R_2(s_1,s_2).\label{Eq2}
\end{eqnarray}

\end{example}

\begin{remark} Equation~(\ref{Eq1}) is satisfied by the self-distributitive property of $*$. On the other hand, equation~(\ref{Eq2}) is satisfied by applying a combination of singquandle properties. Consider equation~(\ref{eq3}) in the definition of a singquandle, let $b = c$. Therefore, we obtain 
\begin{eqnarray*}
(c \bar{*} R_1(a,c))* a &=& (c * R_2(a,c)) \bar{*} c,
\end{eqnarray*}
which can be written as
\begin{eqnarray*}
((c \bar{*} R_1(a,c))* a)*c &=& c * R_2(a,c).
\end{eqnarray*}
Now, we let $w = c \bar{*} R_1(a,c) \iff w * R_1(a,c) = c$. Next, we make the appropriate substitution to obtain 
\[ (w *a) * c  = (w * R_1(a,c)) * R_2(a,c).\]
Lastly, let $w=x$, $a = s_1$, and $b = s_2$ to obtain 
\[ (x * s_1)*s_2 = (x * R_1 (s_1,s_2)) * R_2(s_1,s_2). \]
\end{remark}


Let $D$ be a diagram of an oriented singular link $L$. We will denoted the set of arcs of $D$ by $\mathcal{A}(D)$ and the connected regions of $\mathbb{R}^2\setminus D$ by $\mathcal{R}(D)$. Using the notion of a singquandle homomorphism given in Definition~\ref{singHom}, we have the following notion of colorings by singquandles. 

\begin{definition}
Let $(S,*, R_1, R_2)$ be an oriented singquandle. An \emph{$S$-coloring} of $D$ is a map $f: \mathcal{A}(D) \rightarrow S$ such that at a crossing with  $u_1, u_2, o_1 \in \mathcal{A}(D)$ and at a singular crossing with $a_1,a_2,a_3,a_4 \in \mathcal{A}(D)$ the following conditions are satisfied,
\begin{eqnarray}
f(u_2) = f(u_1) * f(o_1), \\
f(a_3) = R_1(f(a_1),f(a_2)), \\
f(a_4) = R_2(f(a_1),f(a_2)).
\end{eqnarray}  
The conditions above are illustrated in Figure~\ref{shadowrule}.
\end{definition}

\begin{definition}
Let $(S,X,*, R_1, R_2, \cdot)$ be a shadow singquandle. An \emph{$(S,X)$-coloring} of $D$ is a map $f \times \phi: \mathcal{A}(D) \times \mathcal{R}(D) \rightarrow S \times X$ satisfying the following conditions,
\begin{itemize}
    \item $f$ is an $S$-coloring of $D$.
    \item $\phi(\mathcal{R}(D)) \subset X$.
    \item For $a \in \mathcal{A}(D)$ and $x_1, x_2 \in \mathcal{R}(D)$ the following
    \begin{equation}  
    \phi(x_1) \cdot f(a) = \phi(x_2).
    \end{equation}
\end{itemize}
The condition above is illustrated in Figure~\ref{shadowrule}.
\end{definition} 

When there is no confusion we will refer to an $(S,X)$-coloring by  a \emph{shadow coloring} of $D$.
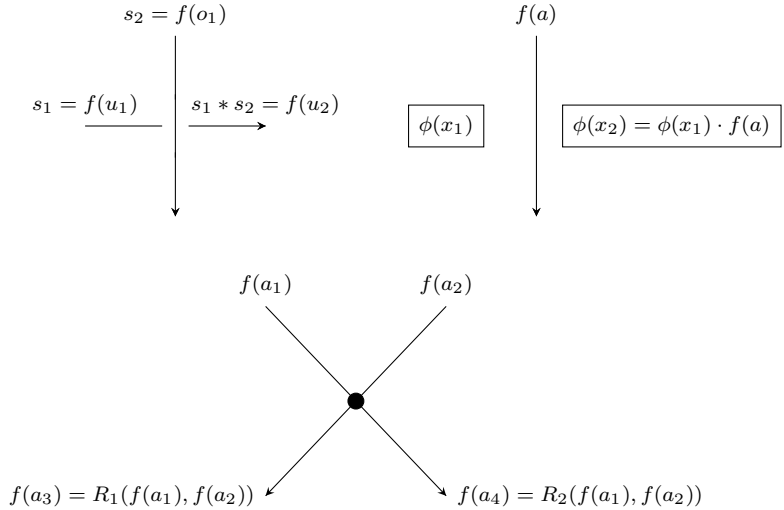
\begin{figure}[h]
\begin{tikzpicture}[use Hobby shortcut,scale=1.2]
\begin{knot}[
consider self intersections=true,
clip width=5,
  ignore endpoint intersections=false,
  flip crossing/.list={3,4,5,6,11,13}
]
\strand[decoration={markings,mark=at position 1 with
    {\arrow[scale=1,>=stealth]{>}}},postaction={decorate}] (0,1) ..(0,-1); 
\strand[decoration={markings,mark=at position 1 with
    {\arrow[scale=1,>=stealth]{>}}},postaction={decorate}] (-1,0) ..(1,0);
    
\strand[decoration={markings,mark=at position 1 with
    {\arrow[scale=1,>=stealth]{>}}},postaction={decorate}] (4,1) ..(4,-1);

\draw (1,-2)..(3,-4.1)[decoration={markings,mark=at position 1 with
    {\arrow[scale=1,>=stealth]{>}}},postaction={decorate}];
\draw[decoration={markings,mark=at position 1 with
    {\arrow[scale=1,>=stealth]{>}}},postaction={decorate}] (3,-2)..(1,-4.1);
   
\end{knot}

\node[above] at (-1,0) {\tiny $s_1=f(u_1)$};
\node[above] at (0,1) {\tiny $s_2=f(o_1)$};
\node[above] at (1,0) {\tiny $s_1 * s_2 =f(u_2)$};

\draw (3,0) node [draw] {\tiny $\phi(x_1)$};
\draw (5.5,0) node [draw] {\tiny $\phi(x_2) = \phi(x_1)\cdot f(a)$};
\node[above] at (4,1) {\tiny $f(a)$};

\node[above] at (1,-2) {\tiny $f(a_1)$};
\node[above] at (3,-2) {\tiny $f(a_2)$};
\node[right] at (3,-4.1) {\tiny $f(a_4)=R_2 (f(a_1),f(a_2))$};
\node[left] at (1,-4.1) {\tiny $f(a_3)=R_1(f(a_1),f(a_2))$};
\node[circle,draw=black, fill=black, inner sep=0pt,minimum size=6pt] (a) at (2,-3.05) {};
\end{tikzpicture}
	\caption{Arcs and regions of diagram $D$.}
		\label{shadowrule}
\end{figure}

%
%
%
%
%
We will denote a region coloring by a box around the shadow element and we will denote an arc coloring by an element of a singquandle without a box. Note that the conditions required for the set $X$ to be an $S$-set for some oriented sinquandle, are the conditions needed to guarantee that shadow colorings are well defined at crossings, see Figure~\ref{shadowX}.

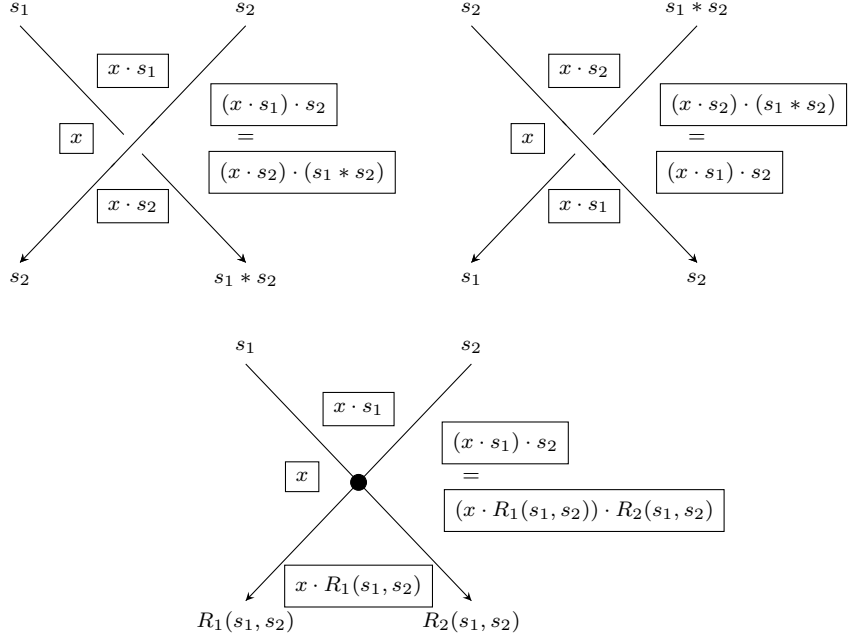
\begin{figure}[h]
\begin{tikzpicture}[use Hobby shortcut,scale=1.5]
\begin{knot}[
consider self intersections=true,
clip width=5,
  ignore endpoint intersections=false,
  flip crossing/.list={3,4,5,6,11,13}
]
\strand[decoration={markings,mark=at position 1 with
    {\arrow[scale=1,>=stealth]{>}}},postaction={decorate}] (1,1) ..(-1,-1.1); 
\strand[decoration={markings,mark=at position 1 with
    {\arrow[scale=1,>=stealth]{>}}},postaction={decorate}] (-1,1) ..(1,-1.1);
    
\strand[decoration={markings,mark=at position 1 with
    {\arrow[scale=1,>=stealth]{>}}},postaction={decorate}] (3,1) ..(5,-1.1);
\strand[decoration={markings,mark=at position 1 with
    {\arrow[scale=1,>=stealth]{>}}},postaction={decorate}] (5,1) ..(3,-1.1);

\draw[decoration={markings,mark=at position 1 with
    {\arrow[scale=1,>=stealth]{>}}},postaction={decorate}] (1,-2) ..(3,-4.1);
\draw[decoration={markings,mark=at position 1 with
    {\arrow[scale=1,>=stealth]{>}}},postaction={decorate}] (3,-2) ..(1,-4.1);     
\end{knot}

\node[above] at (-1,1) {\tiny $s_1$};
\node[above] at (1,1) {\tiny $s_2$};

\node[below] at (-1,-1.1) {\tiny $s_2$};
\node[below] at (1,-1.1) {\tiny $s_1 *s_2$};

\draw (-.5,0) node [draw] {\tiny $x$};
\draw (0,.6) node [draw] {\tiny $x \cdot s_1$};
\draw (1.25,.3) node [draw] {\tiny $(x \cdot s_1) \cdot s_2$};
\draw  (0,-.6) node [draw] {\tiny $x \cdot s_2$};
\draw (1.5,-.3) node [draw] {\tiny $(x \cdot s_2) \cdot (s_1 * s_2)$};
\node at (1,0) {\tiny $=$};
\node[below] at (3,-1.1) {\tiny $s_1$};
\node[below] at (5,-1.1) {\tiny $s_2$};

\node[above] at (3,1) {\tiny $s_2$};
\node[above] at (5,1) {\tiny $s_1 * s_2$};

\draw (3.5,0) node [draw] {\tiny $x$};
\draw (4,.6) node [draw] {\tiny $x \cdot s_2$};
\draw (5.5,.3) node [draw] {\tiny $(x \cdot s_2) \cdot (s_1 * s_2)$};
\draw (4,-.6) node [draw] {\tiny $x \cdot s_1$};
\draw (5.2,-.3) node [draw] {\tiny $(x \cdot s_1) \cdot s_2$};
\node at (5,0) {\tiny $=$};

\node[above] at (1,-2) {\tiny $s_1$};
\node[above] at (3,-2) {\tiny $s_2$};

\node[below] at (1,-4.1) {\tiny $R_1(s_1,s_2)$};
\node[below] at (3,-4.1) {\tiny $R_2(s_1,s_2)$};

\draw (1.5,-3) node [draw] {\tiny $x$};
\draw (2,-2.4) node [draw] {\tiny $x \cdot s_1$};
\draw (3.3,-2.7) node [draw] {\tiny $(x \cdot s_1) \cdot s_2$};
\draw (2,-3.97) node [draw] {\tiny $x \cdot R_1(s_1,s_2)$};
\draw (4,-3.3) node [draw] {\tiny $(x \cdot R_1(s_1,s_2)) \cdot R_2(s_1,s_2) $};
\node at (3,-3) {\tiny $=$};
\node[circle,draw=black, fill=black, inner sep=0pt,minimum size=6pt] (a) at (2,-3.05) {};
\end{tikzpicture}
	\caption{Shadow coloring at positive, negative, and singular crossings.}
		\label{shadowX}
\end{figure}

\begin{proposition}
Let $L$ be a singular link diagram and $(S,X,*,R_1,R_2,\cdot)$ be a singquandle shadow. Then for each singquandle coloring of $L$ by $S$ and each element of $X$ there is exactly one shadow coloring of $L$.
\end{proposition}

\begin{proof}
Consider a singquandle coloring  of $L$ and choose a region of $L$. Any element of $X$ can be assigned to the chosen region, and any choice determines a unique shadow color of each region by following the rule in Figure~\ref{shadowrule}.
\end{proof}

\begin{definition}
 Let $L$ be singular link diagram and $(S,X)$ is a singquandle shadow. The \emph{shadow counting invariant}, $\#\textup{Col}_{(S,X)}(L)$, is the number of shadow colorings of $L$ by $(S,X)$.
\end{definition}

\begin{example}
Let $(S,*,R_1,R_2)$ be the singquandle with $S=\mathbb{Z}_{10}$ and operations defined by $x*y = 3x-2y$, $x \bar{*} y = 7 x - 6 y$, $R_1(x,y) = 4x+6y$, and $R_2(x,y) = 8x+2y$. We can define a shadow structure by $X=\mathbb{Z}_4$ with the map $\cdot s: \mathbb{Z}_4  \rightarrow \mathbb{Z}_4$ for each $s \in S$ defined by $x\cdot s= 2+x + 2x^2$. We also represent this shadow operation by the shadow matrix
\[
\left[
\begin{array}{cccccccccc}
 1 & 1 & 1 & 1 & 1 & 1 & 1 & 1 & 1 & 1 \\
 0 & 0 & 0 & 0 & 0 & 0 & 0 & 0 & 0 & 0 \\
 3 & 3 & 3 & 3 & 3 & 3 & 3 & 3 & 3 & 3 \\
 2 & 2 & 2 & 2 & 2 & 2 & 2 & 2 & 2 & 2 \\
\end{array}
\right].
\]
We will compute a shadow coloring for the following \emph{2 bouquet graph of type $K$} listed as $3_1^k$ in \cite{Oyamaguchi}.
\begin{figure}[h]
\begin{tikzpicture}[scale=.6,use Hobby shortcut,add arrow/.style={postaction={decorate}, decoration={
  markings,
  mark=at position 1 with {\arrow[scale=2,>=stealth]{<}}}}]
\begin{knot}[
  consider self intersections=true,
  ignore endpoint intersections=false,
 flip crossing/.list={4},
  clip width=5,
  only when rendering/.style={
  }
  ]
\strand ([closed]0,1.5)..(-1.2,-1.5).. (3,-3.5) ..(0,-1.2) ..(-3,-3.5) ..(1.2,-1.5)..(0,1.5)[add arrow];
\end{knot}

\node[circle,draw=black, fill=black, inner sep=0pt,minimum size=8pt] (a) at (1.24,-1.35) {};
\node[above] at (0,-1) {\tiny $s_2$};
\node[above] at (0,1.5) {\tiny $s_4$};
\node[right] at (3,-3.5) {\tiny $s_3$};
\node[left] at (-3,-3.5) {\tiny $s_1$};

\draw (0,-2) node [draw] {\tiny $x_i$};
\draw (-2,-3) node [draw] {\tiny $x_i \cdot s_4$};
\draw (2,-3) node [draw] {\tiny $x_i \cdot s_1$};
\draw (0,0) node [draw] {\tiny $x_i  \cdot s_2$};
\draw (-5,-2.8) node [draw] {\tiny $(x_i \cdot s_4) \cdot s_1$};
\draw (5,-2.8) node [draw] {\tiny $(x_i \cdot s_1) \cdot s_3$};
\draw (0,2.5) node [draw] {\tiny $(x_i  \cdot s_2) \cdot s_4$};

\end{tikzpicture}
		\caption{Shadow coloring of $3_1^k$.}
		\label{shadowdiagram}
\end{figure}
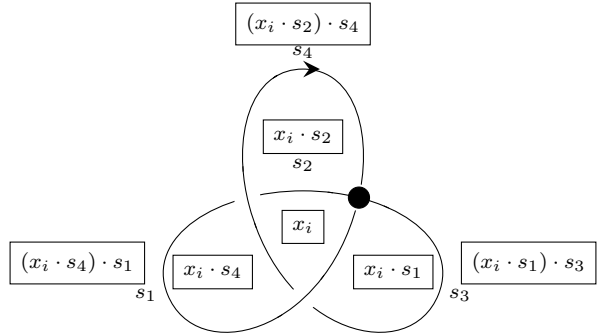
This singular knot has one coloring by $S$ given by
\[\textup{Hom}(\mathcal{SQ}(3_1^k),S) =\lbrace (s_1\to 0,s_2 \to 0, s_3 \to 0, s_4 \to 0)\rbrace. \]
For this coloring we also obtain one shadow coloring for each element of $X$. Therefore, we have the following four shadow colorings by the above shadow singquandle.

\begin{figure}[h]
\centering
\begin{tikzpicture}[scale=.4,use Hobby shortcut,add arrow/.style={postaction={decorate}, decoration={
  markings,
  mark=at position 1 with {\arrow[scale=2,>=stealth]{<}}}}]
\begin{knot}[
  consider self intersections=true,
  ignore endpoint intersections=false,
 flip crossing/.list={4},
  clip width=5,
  only when rendering/.style={
  }
  ]
\strand ([closed]0,1.5)..(-1.2,-1.5).. (3,-3.5) ..(0,-1.2) ..(-3,-3.5) ..(1.2,-1.5)..(0,1.5)[add arrow];
\end{knot}

\node[circle,draw=black, fill=black, inner sep=0pt,minimum size=8pt] (a) at (1.24,-1.35) {};
\node[above] at (0,-1) {\tiny $0$};
\node[above] at (0,1.5) {\tiny $0$};
\node[right] at (3,-3.5) {\tiny $0$};
\node[left] at (-3,-3.5) {\tiny $0$};

\draw (0,-2) node [draw] {\tiny $1$};
\draw (-2,-3) node [draw] {\tiny $1$};
\draw (2,-3) node [draw] {\tiny $1$};
\draw (0,.5) node [draw] {\tiny $1$};
\draw (-4.5,-3) node [draw] {\tiny $1$};
\end{tikzpicture}
\hspace{.2cm}
\begin{tikzpicture}[scale=.4,use Hobby shortcut,add arrow/.style={postaction={decorate}, decoration={
  markings,
  mark=at position 1 with {\arrow[scale=2,>=stealth]{<}}}}]
\begin{knot}[
  consider self intersections=true,
  ignore endpoint intersections=false,
 flip crossing/.list={4},
  clip width=5,
  only when rendering/.style={
  }
  ]
\strand ([closed]0,1.5)..(-1.2,-1.5).. (3,-3.5) ..(0,-1.2) ..(-3,-3.5) ..(1.2,-1.5)..(0,1.5)[add arrow];
\end{knot}

\node[circle,draw=black, fill=black, inner sep=0pt,minimum size=8pt] (a) at (1.24,-1.35) {};
\node[above] at (0,-1) {\tiny $0$};
\node[above] at (0,1.5) {\tiny $0$};
\node[right] at (3,-3.5) {\tiny $0$};
\node[left] at (-3,-3.5) {\tiny $0$};

\draw (0,-2) node [draw] {\tiny $2$};
\draw (-2,-3) node [draw] {\tiny $0$};
\draw (2,-3) node [draw] {\tiny $0$};
\draw (0,.5) node [draw] {\tiny $0$};
\draw (-4.5,-3) node [draw] {\tiny $2$};

\end{tikzpicture}
\hspace{.2cm}
\begin{tikzpicture}[scale=.4,use Hobby shortcut,add arrow/.style={postaction={decorate}, decoration={
  markings,
  mark=at position 1 with {\arrow[scale=2,>=stealth]{<}}}}]
\begin{knot}[
  consider self intersections=true,
  ignore endpoint intersections=false,
 flip crossing/.list={4},
  clip width=5,
  only when rendering/.style={
  }
  ]
\strand ([closed]0,1.5)..(-1.2,-1.5).. (3,-3.5) ..(0,-1.2) ..(-3,-3.5) ..(1.2,-1.5)..(0,1.5)[add arrow];
\end{knot}

\node[circle,draw=black, fill=black, inner sep=0pt,minimum size=8pt] (a) at (1.24,-1.35) {};
\node[above] at (0,-1) {\tiny $0$};
\node[above] at (0,1.5) {\tiny $0$};
\node[right] at (3,-3.5) {\tiny $0$};
\node[left] at (-3,-3.5) {\tiny $0$};

\draw (0,-2) node [draw] {\tiny $3$};
\draw (-2,-3) node [draw] {\tiny $3$};
\draw (2,-3) node [draw] {\tiny $3$};
\draw (0,.5) node [draw] {\tiny $3$};
\draw (-4.5,-3) node [draw] {\tiny $3$};

\end{tikzpicture}
\hspace{.2cm}
\begin{tikzpicture}[scale=.4,use Hobby shortcut,add arrow/.style={postaction={decorate}, decoration={
  markings,
  mark=at position 1 with {\arrow[scale=2,>=stealth]{<}}}}]
\begin{knot}[
  consider self intersections=true,
  ignore endpoint intersections=false,
 flip crossing/.list={4},
  clip width=5,
  only when rendering/.style={
  }
  ]
\strand ([closed]0,1.5)..(-1.2,-1.5).. (3,-3.5) ..(0,-1.2) ..(-3,-3.5) ..(1.2,-1.5)..(0,1.5)[add arrow];
\end{knot}

\node[circle,draw=black, fill=black, inner sep=0pt,minimum size=8pt] (a) at (1.24,-1.35) {};
\node[above] at (0,-1) {\tiny $0$};
\node[above] at (0,1.5) {\tiny $0$};
\node[right] at (3,-3.5) {\tiny $0$};
\node[left] at (-3,-3.5) {\tiny $0$};

\draw (0,-2) node [draw] {\tiny $0$};
\draw (-2,-3) node [draw] {\tiny $2$};
\draw (2,-3) node [draw] {\tiny $2$};
\draw (0,.5) node [draw] {\tiny $2$};
\draw (-4.5,-3) node [draw] {\tiny $0$};

\end{tikzpicture}
		\caption{Shadow colorings of $3_1^k$ by singquandle shadow $X$.}
		\label{test}
\end{figure}
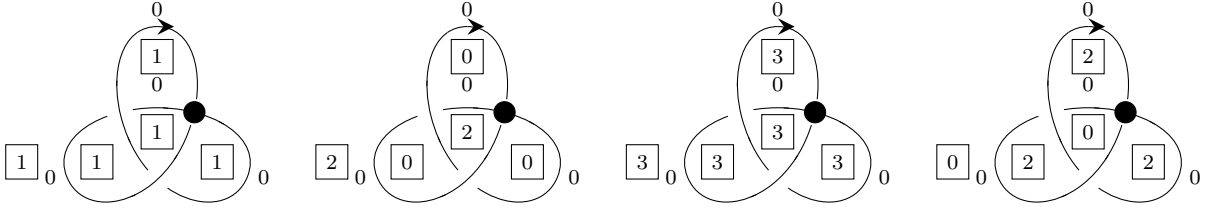
\noindent Therefore, $\#\textup{Col}_{(S,X)}(3_1^k) = 4.$
\end{example}

The following corollary implies that the shadow counting invariant does not contain any information not contained by the singquandle counting invariant. We obtain the following result by noticing that for each coloring of an oriented singular link by an oriented singquandle we have a different shadow coloring for each element in the $S$-set. 

\begin{corollary}\label{singandshadow}
The shadow counting invariant of a singular link $L$ by the $S$-shadow $(S,X)$ is given by 
\[\# \textup{Col}_{(S,X)}(L) = \vert X \vert \, \# \textup{Col}_S(L) ,\]
where $\#\textup{Col}_S(L)$ is the singquandle counting invariant.
\end{corollary}

We can define the following polynomial for a sinquandle shadow to obtain a singquandle shadow invariant. 
\begin{definition}
The \emph{shadow singquandle polynomial}, denoted by $\textup{sp}(S,X)$, of the shadow singquandle $(S,X,*,R_1,R_2,\cdot)$ is the sum 
\[ \textup{sp}(S,X) = \sum_{x \in X} t^{r(x)},\]
where $r(x) = \vert \lbrace s \in S \, ; \, x \cdot s = x \rbrace\vert$. Furthermore, If $(S^{\prime},Y)$ is a subshadow of $(S,X)$, then the \emph{subshadow singquandle polynomial} of $(S^{\prime},Y)$ is 
\[ \textup{Subsp} ((S',Y) \subset (S,X)) = \sum_{x \in Y} t^{r(x)}, \]
where $r(x) =\vert \lbrace s' \in S' \, ; \, x \cdot s' = x \rbrace \vert$.
\end{definition}

\begin{example}
Let $S= \mathbb{Z}_6 =\{ 1,2,3,4,5,0\}$ with singquandle operations defined by $x* y = 5x-4y = x\bar{*}y$, $R_1(x,y)=2x+y$ and $R_2(x,y)= 5x+4y$. Consider the shadow structure defined by $X = \mathbb{Z}_2 = \{1,0\}$ with shadow matrix 
\[ \left[
\begin{array}{cccccc}
 1 & 1 & 1 & 1 & 1 & 1 \\
 0 & 0 & 0 & 0 & 0 & 0 \\
\end{array}
\right].\]
 Note that we can compute $r(x)$ for each $x \in X$ by going through the row of the shadow matrix and counting the occurrences of the row number. Therefore, $r(1)=6$ and $r(0) = 6$, and the shadow singquandle polynomial of $(S,X)$ is
\[ \textup{sp}(S,X) = 2t^6. \]
We will consider two types of subshadows. We will first consider a subset of $X$ closed under the action of $S$. When we consider the subshadow $(S, Y) \subset (S,X)$, where $Y = \{1\}$, note that we can check from the shadow matrix that $Y$ is closed under the action of $S$. The subshadow $(S,Y)$ has the following subshadow singquandle polynomial 
\[ \textup{Subsp}(S,Y) = t^6.\]
Next, we consider the subshadow $(S',Y) \subset (S,X)$, where $S'$ is the subsingquandle consisting of $\{2,4,0\}$ and $Y = \{1 \}$. A straightforward computation shows that $S^\prime$ is closed under the singquandle operations. Furthermore, we can check from the shadow matrix that $Y$ is closed under the action of $S^\prime$. The subshadow $(S',Y)$ has the following subshadow singquandle polynomial,
\[ \textup{Subsp}(S^\prime,Y)= t^3. \]
\end{example}
We now prove that the shadow singquandle polynomial is an invariant of shadow singquandles.

\begin{proposition} 
Let $(S,X)$ and $(S',Y)$ be two singquandle shadows.
If $(S,X)$ and $(S',Y)$  are isomorphic, then they have equal shadow polynomials, $\textup{sp}(S,X) = \textup{sp}(S',Y)$.
\end{proposition}

\begin{proof}


Suppose the pair $\phi: X \rightarrow Y$ and $f: S \rightarrow S'$ is a shadow singquandle isomorphism. Then $r(\phi(x)) = r(x)$ and the contribution to $\textup{sp}(S,X)$ from $x \in X$ is the same as the contribution of $\phi(x) \in Y$ to $\textup{sp}(S',Y)$. Since $\phi$ and $f$ are bijective maps satisfying equations~(\ref{3.6}), (\ref{3.7}), (\ref{3.8}) and (\ref{action}), the result follows.
\end{proof}

The shadow polynomial can be used to distinguish and classify shadow singqundles. In the following example, we distinguish two shadow singquandles.
\begin{example}
Let $(S,*,R_1,R_2)$ be a singquandle with $S= \mathbb{Z}_8$, $x*y= 5x-4y = x \bar{*} y$, $R_1(x,y) = 3 x + 4 y$, and $R_2(x,y)= 4x+3y$. Let $X = \mathbb{Z}_4$ with shadow matrix 
\[ \left[
\begin{array}{cccccccc}
 0 & 1 & 0 & 1 & 0 & 1 & 0 & 1 \\
 1 & 2 & 1 & 2 & 1 & 2 & 1 & 2 \\
 2 & 3 & 2 & 3 & 2 & 3 & 2 & 3 \\
 3 & 0 & 3 & 0 & 3 & 0 & 3 & 0 \\
\end{array}
\right]. \]
The $(S,X, *, R_1,R_2,\cdot)$ is a singquandle shadow with shadow polynomial 
\[ \textup{sp}(S,X) = 4t^4. \]
Let $W = \mathbb{Z}_4$ with shadow matrix
\[ \left[
\begin{array}{cccccccc}
 3 & 3 & 3 & 3 & 3 & 3 & 3 & 3 \\
 2 & 2 & 2 & 2 & 2 & 2 & 2 & 2 \\
 1 & 1 & 1 & 1 & 1 & 1 & 1 & 1 \\
 0 & 0 & 0 & 0 & 0 & 0 & 0 & 0 \\
\end{array}
\right]. \]
The $(S,W, *, R_1,R_2,\bullet )$ is a singquandle shadow with shadow polynomial 
\[ \textup{sp}(S,W) = 2+2t^8. \]
\end{example}

We see that the shadow singquandle polynomial is an effective invariant of singquandle shadows. 

In this section, we see that by simply computing the shadow counting invariant of an oriented singular link we do not obtain any more information than that obtained from the singquandle counting invariant. Therefore, in the following section we will enhance the shadow counting invariant in order to obtain a stronger invariant. 

\section{Enhanced Shadow Counting Invariant }\label{enhanc}

In this section, we will jazz up the shadow counting invariant from the previous section. We will combine the $S$-shadow counting invariant and the shadow polynomial in order to define an enhanced shadow singquandle invariant for singular link. 

\begin{definition}
Let $f \times \phi$ be a shadow coloring of an oriented singular link diagram $D$. The closure of the set of shadow colors under the action of the image subsingquandle $\textup{Im}(f) \subset S$ of $f \times \phi$ is a subshadow called the \emph{shadow image} of $f \times \phi$, which we denote by $\textup{om}(f \times \phi)$\footnote{The choice of $\textup{om}(f \times \phi)$ to denote the shadow image of $f \times \phi$ was derived from the french word \emph{ombre} for 
shadow}.
\end{definition}


\begin{definition}
Let $(S,X)$ be an $S$-shadow and let $L$ be an oriented singular link with diagram $D$. The \emph{singquandle shadow polynomial invariant} of $L$ with respect $(S,X)$ is 
\[ SP(L) = \sum_{f \times \phi \in \textup{shadow coloring}} u^{\textup{Subsp}  (\textup{om}(f \times \phi) \subset (S,X))}.\]
\end{definition}

\section{Examples}\label{examples}
In this section, we present two examples in which we show that the shadow sinquandle polynomial is an enhancement of the singquandle counting invariant. In the first example, we include a pair of singular knots with the same singquandle counting invariant and the same singquandle polynomial but are distinguished by the singquandle shadow polynomial invariant.  The computations were performed by \emph{Mathematica} and \emph{python} independently and checked by hand.
\begin{example}
Let $(S,X,*,R_1,R_2,\cdot)$ be the shadow singquandle with $S=\mathbb{Z}_8$, $X=\mathbb{Z}_6$ and operations $x*y = 3x-2y = x \bar{*} y$, $R_1(x,y) = 7x+6y$, $R_2(x,y) = 2x+3y$, and shadow matrix 
\[ \left[
\begin{array}{cccccccc}
 4 & 1 & 4 & 1 & 4 & 1 & 4 & 1 \\
 5 & 2 & 5 & 2 & 5 & 2 & 5 & 2 \\
 0 & 3 & 0 & 3 & 0 & 3 & 0 & 3 \\
 1 & 4 & 1 & 4 & 1 & 4 & 1 & 4 \\
 2 & 5 & 2 & 5 & 2 & 5 & 2 & 5 \\
 3 & 0 & 3 & 0 & 3 & 0 & 3 & 0 \\
\end{array}
\right]. \]

The following two \emph{2 bouquet graph of type $K$} listed as $4_1^k$ and $5_4^k$ in \cite{Oyamaguchi}. We obtain the following coloring equations from the singular knot $4_1^k$, 
\begin{eqnarray*}
s_1 &=& s_6 \bar{*} s_2 =-2 s_2 + 3 s_6 \\ 
s_2 &=& s_5 \bar{*} s_6 = 3 s_5 - 2 s_6 \\ 
s_3 &=& R_1(s_1,s_2) = 7 s_1 + 6 s_2 \\ 
s_4 &=& R_2(s_1,s_2) =2 s_1 + 3 s_2 \\ 
s_5 &=& s_4 \bar{*}s_3 =-2 s_3 + 3 s_4 \\ 
s_6 &=& s_3 \bar{*} s_5 = 3 s_3 - 2 s_5.
\end{eqnarray*}
From these equations we obtain the colorings listed below. In the following list we identify each coloring $f \in \textup{Hom}(\mathcal{SQ}(4_1^k),S)$ with the $6$-tuple $(f(s_1),f(s_2),f(s_3),f(s_4),f(s_5),f(s_6))$.
\begin{equation*}
\begin{split}
\textup{Hom}(\mathcal{SQ}(4_1^k),S)=
\{(1, 7, 1, 7, 3, 5), (1, 3, 1, 3, 7, 5), (2, 2, 2, 2, 2, 2), (2, 6, 2, 6, 6, 2), \\
(3, 5, 3, 5, 1, 7), (3, 1, 3, 1, 5, 7), (4, 0, 4, 0, 0, 4),(4, 4, 4, 4, 4, 4), \\
(5, 3, 5, 3, 7, 1), (5, 7, 5, 7, 3, 1), (6, 6, 6, 6, 6, 6), (6, 2, 6, 2, 2, 6), \\
(7, 1, 7, 1, 5, 3), (7, 5, 7, 5, 1,3), (0, 4, 0, 4, 4, 0), (0, 0, 0, 0, 0, 0)\}.
\end{split}
\end{equation*}
We obtain the following coloring equations from the singular knot $5_4^k$, 
\begin{eqnarray*}
s_1 &=& s_7 \bar{*} s_5 = -2 s_5 + 3 s_7\\ 
s_2 &=& s_4 * s_6 = 3 s_4 - 2 s_6\\ 
s_3 &=& R_1(s_1,s_2) = 7 s_1 + 6 s_2 \\ 
s_4 &=& R_2(s_1,s_2) =2 s_1 + 3 s_2 \\ 
s_5 &=& s_3 \bar{*}s_7 = 3 s_3 - 2 s_7 \\ 
s_6 &=& s_5 * s_2 = 2 s_2 + 3 s_5\\
s_7 &=& s_6 \bar{*}s_3 =-2 s_3 + 3 s_6 .
\end{eqnarray*}
From these equations we obtain the colorings listed below. In the following list we identify each coloring $f \in \textup{Hom}(\mathcal{SQ}(5_4^k),S)$ with the $6$-tuple $(f(s_1),f(s_2),f(s_3),f(s_4),f(s_5),f(s_6))$.
\begin{equation*}
\begin{split}
\textup{Hom}(\mathcal{SQ}(5_4^k),S)=\{
(2, 2, 2, 2, 2, 2, 2), (2, 4, 6, 0, 6, 2, 2), (2, 6, 2, 6, 2, 2, 2), (2, 0, 6, 4, 6, 2, 2), \\
(4, 2, 0, 6, 0, 4, 4), (4, 4, 4, 4, 4, 4, 4), (4, 6, 0, 2, 0, 4, 4), (4, 0, 4, 0, 4, 4, 4),\\
(6, 2, 6, 2, 6, 6, 6), (6, 4, 2, 0, 2, 6, 6), (6, 6, 6, 6, 6, 6, 6), (6, 0, 2, 4, 2, 6, 6), \\
(0, 2, 4, 6, 4, 0, 0), (0, 4, 0, 4, 0, 0, 0), (0, 6, 4, 2, 4, 0, 0), 
(0, 0, 0, 0, 0, 0, 0) \}.
\end{split}
\end{equation*}
Therefore, both singular knots have the same singquandle counting invariant $\#\textup{Col}_S(4_1^k)  = 16 =  \#\textup{Col}_S(5_4^k)$. Therefore, by Theorem~\ref{singandshadow} we obtain that the two singular knots have the same shadow counting invariant $\#\textup{Col}_{(S,X)}(4_1^k)  = 96 = \#\textup{Col}_{(S,X)}(5_4^k)$. Furthermore, the two singular knots have the same singquandle polynomial $\phi_{Ssqp}(4_1^k)=4 u^{s_1^2 t_1^2 s_2^2 t_2^2 s_3 t_3}+4 u^{2 s_1^2 t_1^2 s_2^2 t_2^2 s_3
   t_3}+8 u^{4 s_1^2 t_1^2 s_2^2 t_2^2 s_3 t_3} =\phi_{Ssqp}(5_4^k) $. However, the singquandle shadow polynomial invariant distinguishes the two singular knots:

\begin{figure}[h]
    \centering

\begin{tikzpicture}[use Hobby shortcut,scale=1,add arrow/.style={postaction={decorate}, decoration={
  markings,
  mark=at position .5 with {\arrow[scale=1.5,>=stealth]{<}}}}]
\begin{knot}[
  consider self intersections=true,
clip width=3,
  flip crossing/.list={2,4,6}
]
\strand ([closed]90:2) foreach \k in {1,...,5} { .. (90-360/5+\k*720/5:1.5) .. (90+\k*720/5:2) } (90:2)[add arrow];
\end{knot}

\node[circle,draw=black, fill=black, inner sep=0pt,minimum size=6pt] (a) at (1,1.35) {};

\node[above] at (0,1) {\tiny $s_1$};
\node[above] at (0,2) {\tiny $s_2$};
\node[left] at (1.5,.5) {\tiny $s_3$};
\node[right] at (2,.5) {\tiny $s_4$};
\node[left] at (-1.5,-1.5) {\tiny $s_5$};
\node[left] at (-2,.5) {\tiny $s_6$};

\draw (0,0) node [draw] {\tiny $x$};

\node[below] at (0,-2.5){$SP(4_1^k)= 24 u^{t^2}+24 u^t+48 u^2$};

\end{tikzpicture}
\hspace{1in}
\begin{tikzpicture}[use Hobby shortcut,scale=1,add arrow/.style={postaction={decorate}, decoration={
  markings,
  mark=at position .5 with {\arrow[scale=1.5,>=stealth]{>}}}}]
\begin{knot}[
  consider self intersections=true,
  ignore endpoint intersections=false,
 flip crossing/.list={7,4,9},
  clip width=4,
  only when rendering/.style={
  }
  ]
\strand ([closed]0,2)..(-.5,1.75)..(0,0)..(.5,-.2)..(2,-2)..(-.5,-1)..(-2,1.5)..(0,1.5)..(2,1.5)..(.5,-1)..(-2,-2)..(-.5,-.2)..(0,0)..(.5,1.75)..(0,2)[add arrow];

\end{knot}

\node[circle,draw=black, fill=black, inner sep=0pt,minimum size=6pt] (a) at (0,0) {};
\node[right] at (-1,0) {\tiny $s_1$};
\node[right] at (-1,.5) {\tiny $s_2$};
\node[right] at (.5,0) {\tiny $s_3$};
\node[right] at (.5,.5) {\tiny $s_4$};
\node[left] at (-2,.5) {\tiny $s_5$};
\node[right] at (2,.5) {\tiny $s_6$};
\node[left] at (-2,-1) {\tiny $s_7$};

\draw (0,-.5) node [draw] {\tiny $x$};

\node[below] at (0,-2.5){$SP(5_4^k)=48 u^{t^4}+24 u^{t^2}+24 u^t $};

\end{tikzpicture}    
    \caption{Singular knots $4_1^k$ and $5_4^k$ and corresponding $SP$ invariant.}
    \label{Example1}
\end{figure}
\end{example}

\begin{example}
Let $(S,X,*,R_1,R_2,\cdot)$ with $S=\mathbb{Z}_{12}$, $X=\mathbb{Z}_8$, $x*y = 5x-4y = x \bar{*} y$, $R_1(x,y) = 5x+10y$, $R_2(x,y) = 2x+y$, and shadow matrix 
\[ \left[
\begin{array}{cccccccccccc}
 3 & 7 & 7 & 7 & 3 & 7 & 7 & 7 & 3 & 7 & 7 & 7 \\
 2 & 6 & 6 & 6 & 2 & 6 & 6 & 6 & 2 & 6 & 6 & 6 \\
 1 & 5 & 5 & 5 & 1 & 5 & 5 & 5 & 1 & 5 & 5 & 5 \\
 0 & 4 & 4 & 4 & 0 & 4 & 4 & 4 & 0 & 4 & 4 & 4 \\
 7 & 3 & 3 & 3 & 7 & 3 & 3 & 3 & 7 & 3 & 3 & 3 \\
 6 & 2 & 2 & 2 & 6 & 2 & 2 & 2 & 6 & 2 & 2 & 2 \\
 5 & 1 & 1 & 1 & 5 & 1 & 1 & 1 & 5 & 1 & 1 & 1 \\
 4 & 0 & 0 & 0 & 4 & 0 & 0 & 0 & 4 & 0 & 0 & 0 \\
\end{array}
\right]. \]
We compute the shadow polynomial for the following singular knots derived from the classical trefoil knot

\begin{figure}[h]
\centering
\begin{tikzpicture}[scale=.55,use Hobby shortcut]
\begin{knot}[
  consider self intersections=true,
  ignore endpoint intersections=false,
 flip crossing/.list={4},
  clip width=5,
  only when rendering/.style={
  }
  ]
\strand ([closed]0,1.5)[decoration={markings,mark=at position .5 with
    {\arrow[scale=3,>=stealth]{<}}},postaction={decorate}]..(-1.2,-1.5).. (3,-3.5) ..(0,-1.2) ..(-3,-3.5) ..(1.2,-1.5)..(0,1.5);
\end{knot}

\node[circle,draw=black, fill=black, inner sep=0pt,minimum size=8pt] (a) at (1.24,-1.35) {};
\node[above] at (0,-1) {\tiny $s_1$};
\node[above] at (0,1.5) {\tiny $s_2$};
\node[right] at (3,-3.5) {\tiny $s_3$};
\node[left] at (-3,-3.5) {\tiny $s_4$};

\draw (0,-2) node [draw] {\tiny $x$};
\draw (-2,-3) node [draw] {\tiny $x\cdot s_2$};
\draw (2,-3) node [draw] {\tiny $x \cdot s_4$};
\draw (0,.5) node [draw] {\tiny $x \cdot s_1$};
\draw (-4.5,0) node [draw] {\tiny $(x \cdot s_1)\cdot s_2$};

\node at (0,-5) {$K_1$};

\end{tikzpicture}
\hspace{1in}
\begin{tikzpicture}[scale=.55,use Hobby shortcut]
\begin{knot}[
  consider self intersections=true,
  ignore endpoint intersections=false,
 flip crossing/.list={4},
  clip width=5,
  only when rendering/.style={
  }
  ]
\strand ([closed]0,1.5)[decoration={markings,mark=at position .5 with
    {\arrow[scale=3,>=stealth]{<}}},postaction={decorate}]..(-1.2,-1.5).. (3,-3.5) ..(0,-1.2) ..(-3,-3.5) ..(1.2,-1.5)..(0,1.5);
\end{knot}

\node[circle,draw=black, fill=black, inner sep=0pt,minimum size=8pt] (a) at (-1.24,-1.35) {};
\node[circle,draw=black, fill=black, inner sep=0pt,minimum size=8pt] (a) at (1.24,-1.35) {};

\node[above] at (0,-1) {\tiny $s_1$};
\node[above] at (0,1.5) {\tiny $s_2$};
\node[right] at (3,-3.5) {\tiny $s_3$};
\node[left] at (-3,-3.5) {\tiny $s_4$};
\node[right] at (-.55,-2.9) {\tiny $s_5$};

\draw (0,-2) node [draw] {\tiny $x$};
\draw (-2,-3) node [draw] {\tiny $x\cdot s_5$};
\draw (2,-3) node [draw] {\tiny $x \cdot s_4$};
\draw (0,.5) node [draw] {\tiny $x \cdot s_1$};
\draw (-4.5,0) node [draw] {\tiny $(x \cdot s_1)\cdot s_2$};

\node at (0,-5) {$K_2$};

\end{tikzpicture}

\begin{center}
\begin{tikzpicture}[scale=.55,use Hobby shortcut]
\begin{knot}[
  consider self intersections=true,
  ignore endpoint intersections=false,
 flip crossing/.list={4},
  clip width=5,
  only when rendering/.style={
  }
  ]
\strand ([closed]0,1.5)[decoration={markings,mark=at position .5 with
    {\arrow[scale=3,>=stealth]{<}}},postaction={decorate}]..(-1.2,-1.5).. (3,-3.5) ..(0,-1.2) ..(-3,-3.5) ..(1.2,-1.5)..(0,1.5);
\end{knot}

\node[circle,draw=black, fill=black, inner sep=0pt,minimum size=8pt] (a) at (-1.24,-1.35) {};
\node[circle,draw=black, fill=black, inner sep=0pt,minimum size=8pt] (a) at (1.24,-1.35) {};
\node[circle,draw=black, fill=black, inner sep=0pt,minimum size=8pt] (a) at (0,-3.6) {};

\node[above] at (0,-1) {\tiny $s_1$};
\node[above] at (0,1.5) {\tiny $s_2$};
\node[right] at (3,-3.5) {\tiny $s_3$};
\node[left] at (-3,-3.5) {\tiny $s_4$};
\node[right] at (-.55,-2.9) {\tiny $s_5$};
\node[right] at (1,-2.1) {\tiny $s_6$};

\draw (0,-2) node [draw] {\tiny $x$};
\draw (-2,-3) node [draw] {\tiny $x\cdot s_5$};
\draw (2,-3) node [draw] {\tiny $x \cdot s_6$};
\draw (0,.5) node [draw] {\tiny $x \cdot s_1$};
\draw (-4.5,0) node [draw] {\tiny $(x \cdot s_1)\cdot s_2$};

\node at (0,-5) {$K_3$};

\end{tikzpicture}
\end{center}
\hspace{1in}

		\caption{Colorings of the singular knots derived from the trefoil.}
		\label{test}
\end{figure}
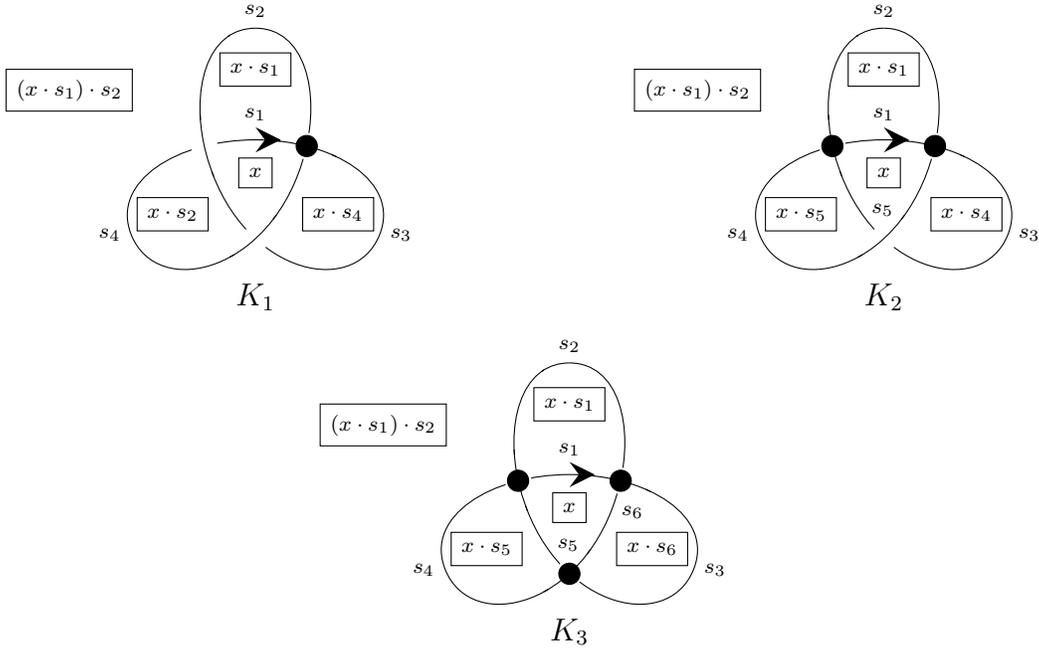

\begin{center}
\begin{tabular}{ c| c| c | l}
$\#\textup{Col}_S$ & $\#\textup{Col}_{(S,X)}$& $SP$ & \text{Singular knot} \\
\hline
4 & 32 & $4 u^{t^2}+4 u^t+24 u^2$ & $K_1$, $K_3$\\
 &     &   $4 u^t+8 u^{2 t}+8 u^3+12 u^2$ & $K_2$
\end{tabular}
\end{center}
In this example we have a collection of singular knots all with the same singquandle counting invariant with respect to $X$. If we then compute the shadow singquandle polynomial invariant we can distinguish $K_2$ from $K_1$ and $K_3$. 
\end{example}

\section{Acknowledgments}
The authors of this paper would like to thank Sam Nelson for fruitful discussion. The authors would also like to thank Hamza Elhamdadi for providing python code to verify the examples in Section 7.


\begin{thebibliography}{2}




\bib{BEHY}{article}{
   author={Bataineh, Khaled},
   author={Elhamdadi, Mohamed},
   author={Hajij, Mustafa},
   author={Youmans, William},
   title={Generating sets of Reidemeister moves of oriented singular links
   and quandles},
   journal={J. Knot Theory Ramifications},
   volume={27},
   date={2018},
   number={14},
   pages={1850064, 15},
   issn={0218-2165},
   review={\MR{3896310}},
}


\bib{BL}{article}{
   author={Birman, Joan S.},
   author={Lin, Xiao-Song},
   title={Knot polynomials and Vassiliev's invariants},
   journal={Invent. Math.},
   volume={111},
   date={1993},
   number={2},
   pages={225--270},
   issn={0020-9910},
   review={\MR{1198809}},
}


\bib{CES}{article}{
   author={Carter, J. Scott},
   author={Elhamdadi, Mohamed},
   author={Saito, Masahico},
   title={Homology theory for the set-theoretic Yang-Baxter equation and
   knot invariants from generalizations of quandles},
   journal={Fund. Math.},
   volume={184},
   date={2004},
   pages={31--54},
   issn={0016-2736},
   review={\MR{2128041}},
}

\bib{CJKLS}{article}{
   author={Carter, J. Scott},
   author={Jelsovsky, Daniel},
   author={Kamada, Seiichi},
   author={Langford, Laurel},
   author={Saito, Masahico},
   title={Quandle cohomology and state-sum invariants of knotted curves and
   surfaces},
   journal={Trans. Amer. Math. Soc.},
   volume={355},
   date={2003},
   number={10},
   pages={3947--3989},
   issn={0002-9947},
   review={\MR{1990571}},
   doi={10.1090/S0002-9947-03-03046-0},
}

\bib{CCEH}{article}{
author={Ceniceros, Jose},
author={Churchill, Indu R. U.},
	author={Elhamdadi, Mohamed},
	author={Hajij, Mustafa},
	title={Cocycle Invariants and Oriented Singular Knots},
	journal={ 	arXiv:2009.07950 (2020), 15 pp.}
	volume={},
	date={},
	number={},
	pages={},
	issn={},
}

\bib{CN}{article}{
   author={Chang, Wesley},
   author={Nelson, Sam},
   title={Rack shadows and their invariants},
   journal={J. Knot Theory Ramifications},
   volume={20},
   date={2011},
   number={9},
   pages={1259--1269},
   issn={0218-2165},
   review={\MR{2844807}},
   doi={10.1142/S0218216511009315},
}

\bib{CCE}{article}{
author={Ceniceros, Jose},
author={Churchill, Indu R. U.},
	author={Elhamdadi, Mohamed},
	title={Polynomial Invariants of Singular Knots and Links},
	journal={to appear in J. Knot Theory Ram.; 	arXiv:2010.13295.}
	volume={},
	date={},
	number={},
	pages={},
	issn={},
	}
\bib{CN}{article}{
   author={Ceniceros, Jose},
   author={Nelson, Sam},
   title={Virtual Yang-Baxter cocycle invariants},
   journal={Trans. Amer. Math. Soc.},
   volume={361},
   date={2009},
   number={10},
   pages={5263--5283},
   issn={0002-9947},
   review={\MR{2515811}},
}


\bib{CEHN}{article}{
author={Churchill, Indu R. U.},
	author={Elhamdadi, Mohamed},
	author={Hajij, Mustafa},
    author={Nelson, Sam},
	title={Singular Knots and Involutive Quandles},
	journal={ J. Knot Theory Ramifications 26 (2017), no. 14, 1750099, 14 pp.}
	volume={},
	date={},
	number={},
	pages={},
	issn={},
	review={\MR{1362791 (96i:57015)}},
}

\bib{E}{article}{
   author={Elhamdadi, Mohamed},
   title={Distributivity in quandles and quasigroups},
   conference={
      title={Algebra, geometry and mathematical physics},
   },
   book={
      series={Springer Proc. Math. Stat.},
      volume={85},
      publisher={Springer, Heidelberg},
   },
   date={2014},
   pages={325--340},
   review={\MR{3275946}},
}

\bib{EFT}{article}{
   author={Elhamdadi, Mohamed},
   author={Fernando, Neranga},
   author={Tsvelikhovskiy, Boris},
   title={Ring theoretic aspects of quandles},
   journal={J. Algebra},
   volume={526},
   date={2019},
   pages={166--187},
   issn={0021-8693},
   review={\MR{3915329}},
}

\bib{EM}{article}{
   author={Elhamdadi, Mohamed},
   author={Moutuou, El-ka\"{\i}oum M.},
   title={Finitely stable racks and rack representations},
   journal={Comm. Algebra},
   volume={46},
   date={2018},
   number={11},
   pages={4787--4802},
   issn={0092-7872},
   review={\MR{3864263}},
}


\bib{EN}{book}{
   author={Elhamdadi, Mohamed},
   author={Nelson, Sam},
   title={Quandles---an introduction to the algebra of knots},
   series={Student Mathematical Library},
   volume={74},
   publisher={American Mathematical Society, Providence, RI},
   date={2015},
   pages={x+245},
   isbn={978-1-4704-2213-4},
   review={\MR{3379534}},
}




	


\bib{Jones}{article}{
   author={Jones, V. F. R.},
   title={Hecke algebra representations of braid groups and link
   polynomials},
   journal={Ann. of Math. (2)},
   volume={126},
   date={1987},
   number={2},
   pages={335--388},
   issn={0003-486X},
   review={\MR{908150}},
}

\bib{Joyce}{article}{
   author={Joyce, David},
   title={A classifying invariant of knots, the knot quandle},
   journal={J. Pure Appl. Algebra},
   volume={23},
   date={1982},
   number={1},
   pages={37--65},
   issn={0022-4049},
   review={\MR{638121}},
}

\bib{JL}{article}{ 
title={An invariant for singular knots},
  author={Juyumaya, Jes{\'u}s},
   author={Lambropoulou, Sofia},
  journal={Journal of Knot Theory and Its Ramifications},
  volume={18},
  number={06},
  pages={825--840},
  year={2009},
  publisher={World Scientific}
}

\bib{Matveev}{article}{
   author={Matveev, S. V.},
   title={Distributive groupoids in knot theory},
   language={Russian},
   journal={Mat. Sb. (N.S.)},
   volume={119(161)},
   date={1982},
   number={1},
   pages={78--88, 160},
   issn={0368-8666},
   review={\MR{672410}},
}

\bib{N2}{article}{
   author={Nelson, Sam},
   title={Generalized quandle polynomials},
   journal={Canad. Math. Bull.},
   volume={54},
   date={2011},
   number={1},
   pages={147--158},
   issn={0008-4395},
   review={\MR{2797975}},
}

\bib{N}{article}{
   author={Nelson, Sam},
   title={A polynomial invariant of finite quandles},
   journal={J. Algebra Appl.},
   volume={7},
   date={2008},
   number={2},
   pages={263--273},
   issn={0219-4988},
   review={\MR{2417045}},
}



	
\bib{Oyamaguchi}{article}{
   author={Oyamaguchi, Natsumi},
   title={Enumeration of spatial 2-bouquet graphs up to flat vertex isotopy},
   journal={Topology Appl.},
   volume={196},
   date={2015},
   number={part B},
   part={part B},
   pages={805--814},
   issn={0166-8641},
   review={\MR{3431017}},
}

\bib{V}{article}{
   author={Vassiliev, V. A.},
   title={Cohomology of knot spaces},
   conference={
      title={Theory of singularities and its applications},
   },
   book={
      series={Adv. Soviet Math.},
      volume={1},
      publisher={Amer. Math. Soc., Providence, RI},
   },
   date={1990},
   pages={23--69},
   review={\MR{1089670}},
}

\end{thebibliography}
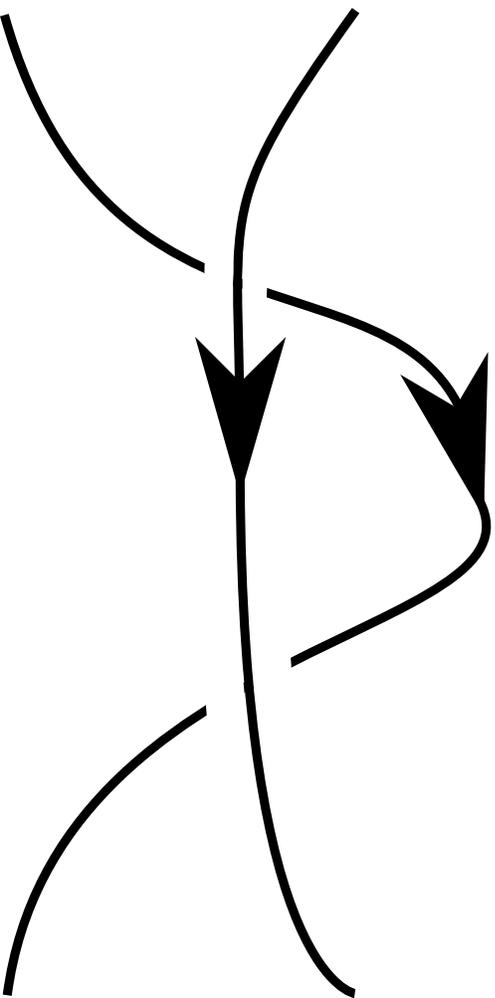
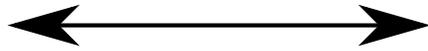
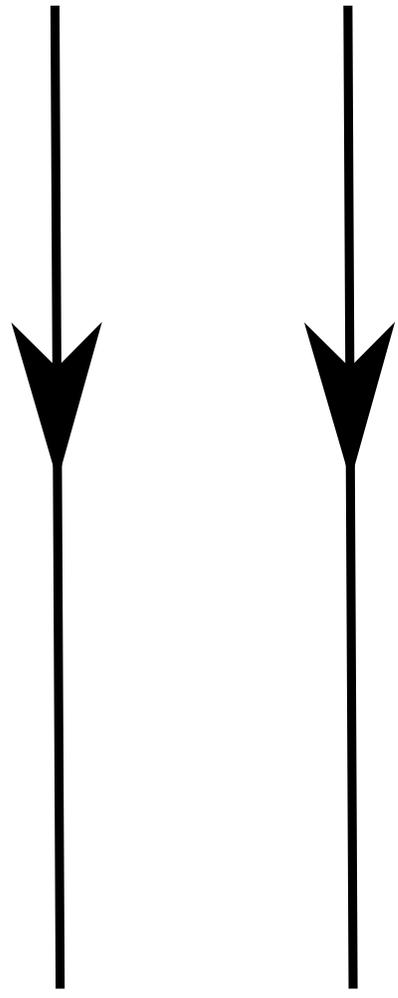
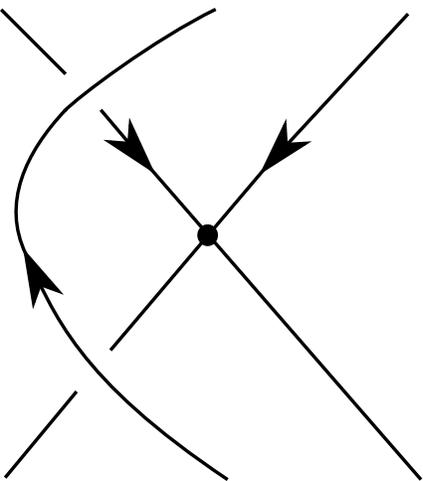
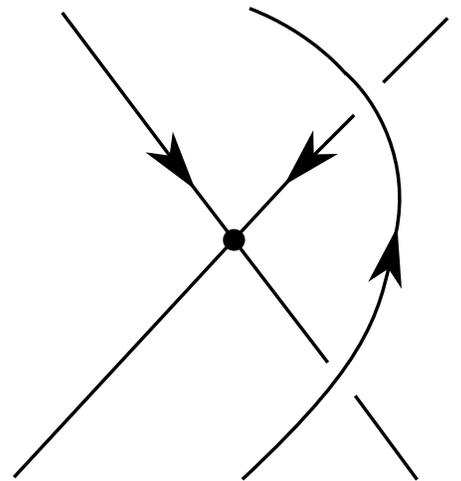
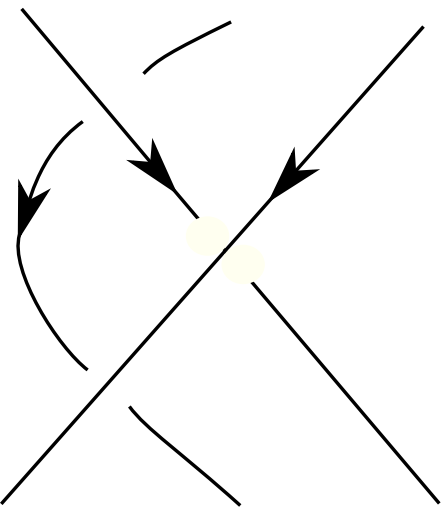
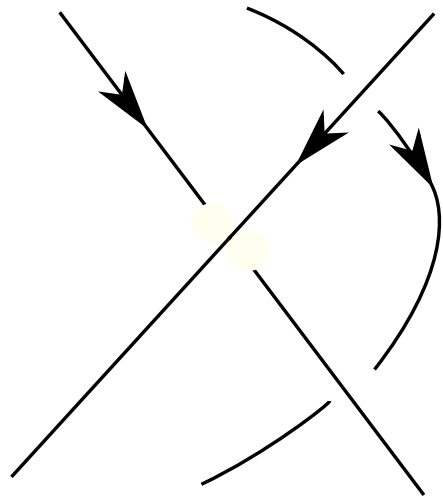
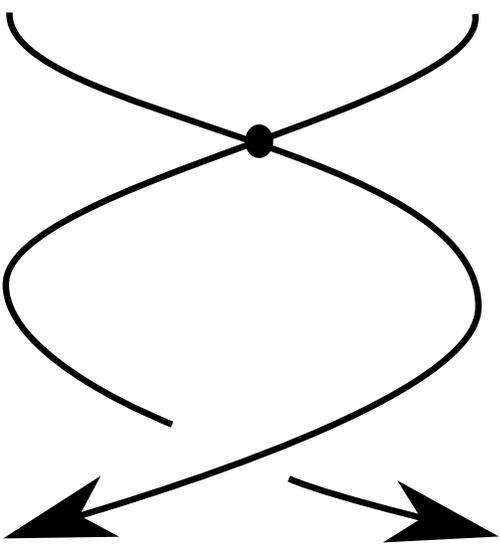
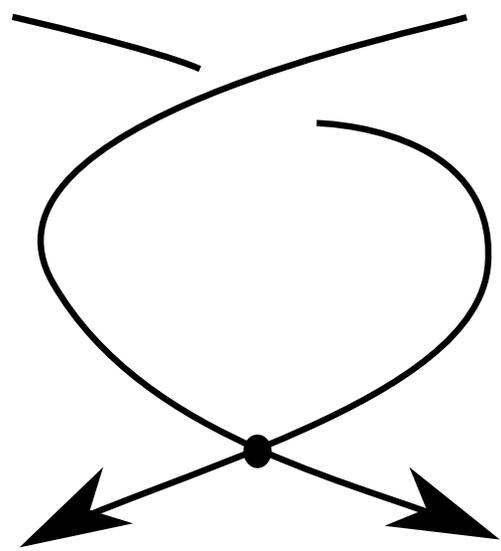
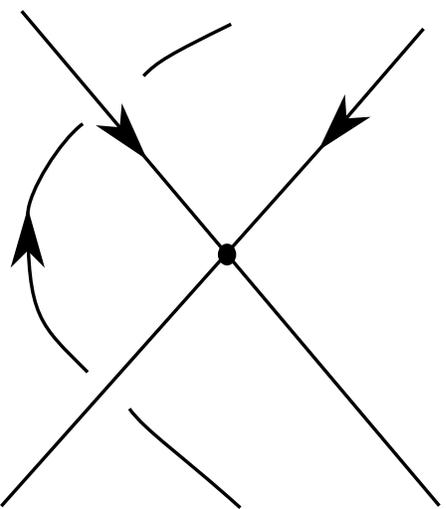
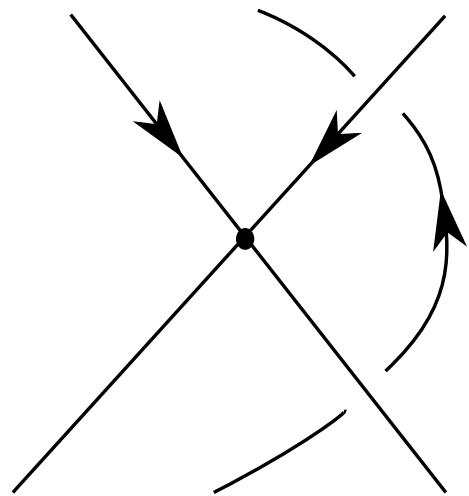
\end{document}